\documentclass[11pt]{article}
\usepackage[width=15.5cm, height=21cm]{geometry}

\usepackage{amsmath}
\usepackage{amssymb, commath, mathtools}
\usepackage{amsfonts, mathrsfs, mathabx}
\usepackage[mathscr]{euscript}
\usepackage{enumitem, cite}

\usepackage{extarrows}
\usepackage{tikz}
\usetikzlibrary{arrows, arrows.meta}
\usepackage{xcolor}
\definecolor{darkgreen}{rgb}{0, 0.5, 0}
\definecolor{lightgreen}{rgb}{0.7, 1, 0.7}

\usepackage{amsthm}
\numberwithin{equation}{section}
\newtheorem{thm}{Theorem}[section]
\newtheorem{cor}[thm]{Corollary}
\newtheorem{lem}[thm]{Lemma}
\newtheorem{prop}[thm]{Proposition}

\theoremstyle{definition}
\newtheorem{example}[thm]{Example}
\newtheorem{defn}[thm]{Definition}
\newtheorem{rem}[thm]{Remark}

\newcommand{\bD}{\mathbb{D}}
\newcommand{\bC}{\mathbb{C}}
\newcommand{\bR}{\mathbb{R}}
\newcommand{\bN}{\mathbb{N}}
\newcommand{\bNz}{\mathbb{N}_0}
\newcommand{\bT}{\mathbb{T}}
\newcommand{\bZ}{\mathbb{Z}}

\newcommand{\cA}{\mathcal{A}}
\newcommand{\cB}{\mathcal{B}}
\newcommand{\cC}{\mathcal{C}}
\newcommand{\cD}{\mathcal{D}}
\newcommand{\cE}{\mathcal{E}}

\newcommand{\cF}{\mathcal{F}}
\newcommand{\cH}{\mathcal{H}}
\newcommand{\cK}{\mathcal{K}}
\newcommand{\cL}{\mathcal{L}}

\newcommand{\cU}{\mathcal{U}}
\newcommand{\cZ}{\mathcal{Z}}

\newcommand{\al}{\alpha}
\newcommand{\be}{\beta}
\newcommand{\ga}{\gamma}

\newcommand{\la}{\lambda}
\newcommand{\La}{\Lambda}
\newcommand{\Tht}{\Theta}
\newcommand{\Om}{\Omega}

\newcommand{\hG}{\widehat{G}}
\newcommand{\hH}{\widehat{H}}
\newcommand{\hnu}{\widehat{\nu}}
\newcommand{\FourierH}{\mathbf{F}}
\newcommand{\tU}{\widetilde{U}}

\renewcommand{\phi}{\varphi}
\newcommand{\enumber}{\operatorname{e}}
\newcommand{\iu}{\operatorname{i}}
\newcommand{\linspan}{\operatorname{span}}

\renewcommand{\Im}{\operatorname{Im}}

\newcommand{\conj}[1]{\overline{#1}}

\newcommand{\charfun}{\mathbf{1}}

\newcommand{\tS}{\widetilde{S}}
\newcommand{\ttS}{\widetilde{\widetilde{S}}}

\usepackage[colorlinks, allcolors=blue]{hyperref}
\newcommand{\doi}[1]{\href{https://doi.org/#1}{doi:#1}}
\newcommand{\myurl}[1]{\href{#1}{#1}}
\newcommand{\eqdef}{\coloneqq}

\allowdisplaybreaks

\title{Integral representation
of translation-invariant operators\\
on reproducing kernel Hilbert spaces}
\author{Shubham R. Bais, Egor A. Maximenko,
D. Venku Naidu}
\begin{document}
\maketitle

\begin{center}
Dedicated to the memory of
Prof. Nikolai Vasilevski (1948--2024)
\end{center}

\begin{abstract}
We suppose that $G$ is a locally compact abelian group,
$Y$ is a measure space, and $H$ is a reproducing kernel Hilbert space on $G\times Y$ such that $H$ is naturally embedded into $L^2(G\times Y)$ and it is invariant under the translations associated with $G$.
We consider the von Neumann algebra of all bounded linear operators acting on $H$ that commute with these translations.
Assuming that this algebra is commutative,
we represent its elements as integral operators and characterize the corresponding integral kernels.
Furthermore, we give W*-algebra structure on the functions associated with the integral kernels.
We apply this general scheme to a series of examples,
including rotation- or translation-invariant operators
in Bergman or Fock spaces.


% Linear operators in reproducing-kernel Hilbert spaces
% C∗-algebras and W∗-algebras in relation to group representations
% Fourier and Fourier-Stieltjes transforms on locally compact and other abelian groups
% Applications of operator theory to differential and integral equations 
% General theory of von Neumann algebras

\end{abstract}

\bigskip
\tableofcontents

\clearpage
\section{Introduction}

Given a reproducing kernel Hilbert space (in short, RKHS) $H$ over a domain $X$,
the bounded linear operators on $H$ can be written via the inner product
(see Aronszajn~\cite[Part~I, Section~11]{Aronszajn1950}).
If $\mu$ is a measure on $X$
and the inner product in $H$
is inherited from $L^2(X,\mu)$,
then the bounded linear operators on $H$
can be written as integral operators
(see Section~\ref{sec:integral_form_of_bounded_linear_operators} below).
Recently, Bais and Venku Naidu~\cite{BaisNaidu2023Fock,
BaisNaidu2023Toeplitz, BaisNaidu2024Bergman},
Bais, Venku Naidu, and Pinlodi~\cite{BaisNaiduPinlodi2023Bergman}, Thagavelu \cite{Thangavelu2024FockSobolev} applied this idea to several examples of RKHSs of analytic functions and several classes of operators invariant under some group actions on the domains.
In this paper, we generalize these ideas to a more abstract setting proposed by Herrera-Ya\~{n}ez, Maximenko, and Ramos-V\'{a}zquez~\cite{HerreraMaximenkoRamos2022}.
Namely, we suppose that $X$ is of the form $G\times Y$,
where $G$ is a locally compact abelian group
with Haar measure $\nu$
and $Y$ is a measure space with measure $\la$.
We consider the unitary representation $\rho$ of $G$ in the space $H$,
consisting of the ``horizontal translations'':
\[
(\rho(a)f)(u,v) \eqdef f(u-a,v)\qquad(a\in G,\ f\in H,\ u\in G,\ v\in Y),
\]
and we denote by $\cC(\rho)$ its centralizer.
We assume some additional conditions;
in particular,
they guarantee the commutativity of the von Neumann algebra $\cC(\rho)$.
Using the Fourier transform of the reproducing kernel ``in the horizontal direction'',
the authors of~\cite{HerreraMaximenkoRamos2022}
constructed a measure space $\Om$
and an isometric isomorphism $R\colon H\to L^2(\Om)$,
and represented all operators belonging to $\cC(\rho)$
in the form $V_b=R^\ast M_b R$,
where $b\in L^\infty(\Om)$
and $M_b$ is the multiplication operator.
In this paper, we represent the elements of $\cC(\rho)$
as integral operators of the following form:
\[
(S_\psi f)(x,y)
=
\int_{G\times Y} \psi(x-u,y,v)\,f(u,v)\,
\dif\nu(u)\,\dif\la(v).
\]
We denote by $\cA$ the set of all functions $\psi\colon G\times Y\times Y\to\bC$
from this integral representation,
and describe bijective correspondences between $b$ in $L^\infty(\Om)$, $\psi$ in $\cA$, and $S_\psi=V_b$ in $\cC(\rho)$.

Furthermore, using these bijections,
we provide $\cA$ with the structure of W*-algebra.
In other words, $\cA$ is a C*-algebra isometrically isomorphic to the von Neumann algebra $\cC(\rho)$.
The ``multiplication'' on $\cA$ is a mixture of the ``box operation''
(the ``composition'' of the integral kernels of integral operators)
and the convolution;
it can also be referred to as a ``partial convolution''.

Thereby, we generalize the following series of pioneering papers where non-trivial C*-algebras of analytic functions were constructed on various subdomains of $\bC$.
Ma and Zhu~\cite{Ma-Zhu_PAMS_2024} gave a concrete non-trivial example of a C*-algebra (in fact, a maximal abelian W*-algebra) of entire functions on the complex plane.  These functions are closely related to the kernels of the integral operators studied by Cao et al. \cite{Cao_SingInt_AdvMath_2020} and Zhu \cite{Zhu_SingInt_IEOT_2015}.  Additionally, the authors of \cite{Ma-Zhu_PAMS_2024} posed an open problem regarding the construction of non-trivial C*-algebras of analytic functions on various subdomains of $\bC$,
especially over the open unit disk $\bD$.
In this direction, Bais~\cite{Bais_PAMS_2024}
constructed two non-trivial class of C*-algebras of analytic functions on the domains $\Pi=\{z\in \bC\colon \Im(z)>0\}$ and $\bC\setminus\{0\}$, respectively.  The C*-algebra over the upper half-plane constructed in \cite{Bais_PAMS_2024} can be translated to get a non-trivial C*-algebra of analytic functions over the open unit disk $\bD$, which answers the open problem posed in
\cite{Ma-Zhu_PAMS_2024}.
We also refer to the works of Mohan and Venku Naidu~\cite{Mohan-Venku_2024} which provide another example of a commutative C*-algebra of analytic functions over $\bD$.  

Recently, Thangavelu~\cite{Thangavelu2024TwistedHeisenberg} constructed a noncommutative algebra of entire functions on $\bC^{2n}$,
which is related to the integral operators studied by Garg and Thangavelu~\cite{GargThangavelu2024TwistedFock} on the twisted Fock space.  

In a series of independent papers,
Karapetyants and Samko~\cite{Karapetyants_2020_convolution_disk},
Karapetyants and Morales~\cite{Karapetyants_2022_convolution_disk},
and Karapetyants and Vagarshakyan~\cite{Karapetyants_2024_convolution_UHP}
also studied similar classes of operators 
and integral kernels in the Bergman spaces over the unit disk
and the upper half-plane,
but they did not pay attention to the C*-algebra structure.

In this article, we give a way to construct commutative C*-algebras (in fact, W*-algebras) of functions on domains of the form $G\times Y\times Y$, where $G$ is a locally compact abelian group and $Y$ is a sigma-finite measure space.
As a consequence, C*-algebras of analytic functions constructed in \cite{Bais_PAMS_2024, Zhu_SingInt_IEOT_2015, Mohan-Venku_2024} can be recovered. 

%Recently, in \cite{Ma-Zhu_PAMS_2024}, a concrete non-trivial example of a C*-algebra of analytic functions is constructed using kernels of singular integral operators studied in \cite{Cao_SingInt_AdvMath_2020, BaisNaidu2023Fock}.  Also, in \cite{Ma-Zhu_PAMS_2024}, the authors posed an interesting question of costructing C*-algebras of analytic functions on other domains.  In this direction, Bais in \cite{Bais_PAMS_2024}, constructed two non-trivial class of C*-algebras of analytic functions on the domains $\Pi=\{z\in \mathbb{C}: \mathrm{Im}(z)>0\}$ and $\mathbb{C}\setminus\{0\}$, respectively.  We refer to \cite{Mohan-Venku_2024} for examples of C*-algebra of analytic functions on the unit disk $\mathbb{D}$.  In this article, we give a way to construct commutative C*-algebra of functions on domain of the form $G\times Y\times Y$, where $G$ is a locally compact abelian group and $Y$ is a sigma-finite measure space.  As a consequence, we recover the C*-algebras of analytic functions constructed in \cite{Bais_PAMS_2024, Zhu_SingInt_IEOT_2015, Mohan-Venku_2024}. 

The rest of the paper has the following structure.
In Section~\ref{sec:integral_form_of_bounded_linear_operators}, we recall the elementary idea that all bounded linear operators acting in $H$ can be written as integral operators.
In Section~\ref{sec:translation_invariant},
we recall the necessary information from \cite{HerreraMaximenkoRamos2022}.
In Section~\ref{sec:integral_representation},
we represent the translation-invariant operators in the integral form and prove the first main result of this paper (Theorem~\ref{thm:V_integral_representation}).
In Section~\ref{sec:cA},
we study the set $\cA$ of all functions $\psi$ from the integral representation above,
and study the bijections between $\cA$, $L^\infty(\Om)$, and $\cC(\rho)$.
In Section~\ref{sec:integral_representation_special_operators},
we give an integral form for some special subclasses of operators belonging to $\cC(\rho)$:
for the operators $\rho(a)$
and for the Toeplitz operators with defining symbols belonging to $L^\infty(G\times Y)$ and invariant under translations by elements of $G$.
In Section \ref{sec:algebra_of_functions},
we define new multiplication operation, adjoint operation and norm on $\cA$, with respect to which it forms a C*-algebra.
At last, in Section~\ref{sec:examples},
we consider a series of examples of operator algebras corresponding to special choices of $G$, $Y$, and $H$, and provide integral representations for the operators belonging to $\cC(\rho)$.
We also consider some examples that can be reduced to our scheme by weighted changes of variables.

\section{Integral form of bounded linear operators in RKHS}
\label{sec:integral_form_of_bounded_linear_operators}

In this section,
we prove that if an RKHS
$H$ is naturally embedded in some $L^2$-space,
then every bounded linear operator acting in $H$
can be represented as an integral operator
(Proposition~\ref{prop:bounded_operators_as_integral_operators}).
Moreover, under an additional assumption
on the integral kernel,
this representation is unique
(Proposition~\ref{prop:integral_operator_uniqueness}).
These facts are corollaries of
Aronszajn~\cite[Part~I, Section~11]{Aronszajn1950}.
For the sake of completeness, we give detailed proofs.

In the particular case if $H$ is the Fock space,
Folland~\cite[Proposition (1.68)]{Folland1989} represented the bounded linear operators on $H$ as integral operators,
but we have not found a general version of this idea.

In this section,
we suppose that $X$ is a nonempty set,
$\mu$ is a measure on $X$,
$H$ is an RKHS over $X$,
such that its inner product is inherited from $L^2(X,\mu)$:
\[
\langle f,g\rangle
= \int_X f\,\conj{g}\,\dif\mu.
\]
We briefly write $L^2(X)$ instead of $L^2(X,\mu)$.
We denote by $(K_w)_{w\in X}$ the reproducing kernel of $H$.

Let $\cB(H)$ be the algebra of all bounded linear operators on $H$, with the usual operator norm:
\begin{equation}
\label{eq:operator_norm}
\|T\| \eqdef \sup_{\substack{f\in H\\\|f\|=1}}\|Tf\|.
\end{equation}

\begin{rem}[dot notation for functions with fixed arguments]
\label{rem:dot_notation_for_functions_with_fixed_arguments}
Given a function $L\colon X^2\to\bC$ and points $a,b$ in $X$, 
we denote by $L(a,\cdot)$ and $L(\cdot,b)$ the functions obtained from $L$ by fixing one of its arguments;
that is, $L(a,\cdot)\colon X\to\bC$,
$L(\cdot,b)\colon X\to\bC$,
\[
L(a,\cdot)(w)\eqdef L(a,w)\quad(w\in X),\qquad
L(\cdot,b)(z)\eqdef L(z,b)\quad(z\in X).
\]
We also extend this notation for functions of three and four arguments.
\end{rem}

\begin{prop}
\label{prop:bounded_operators_as_integral_operators}
Let $T\in\cB(H)$.
Define $K_T\colon X^2\to\bC$ by
\begin{equation}
\label{eq:integral_kernel_associated_to_operator_in_RKHS}
K_T(z,w)
\eqdef (T K_w)(z)
\qquad(w,z\in X).
\end{equation}
Then, for every $f$ in $H$
and every $z$ in $X$,
\begin{equation}
\label{eq:integral_representation_of_bounded_operator}
(T f)(z)
=
\int_X K_T(z,w) f(w) \dif\mu(w).
\end{equation}
Moreover,
the integral kernel $K_T$ has the following properties:
\begin{align}
\label{eq:KT1_in_H}
K_T(\cdot, w) \in H\qquad(w\in X),
\\
\label{eq:KT2_in_H}
\conj{K_T(z, \cdot)} \in H\qquad(z\in X).
\end{align}
\end{prop}

\begin{proof}
Since $H$ is an RKHS,
$T$ can be written via the inner product:
\begin{equation}
\label{eq:bounded_operator_via_inner_product}
(Tf)(z)
= \langle Tf,K_z\rangle
= \langle f, T^\ast K_z \rangle.
\end{equation}
Furthermore,
we apply the assumption
that the inner product in $H$
can be written in the integral form:
\[
(Tf)(z)
=
\langle f, T^\ast K_z \rangle
=
\int_X f(w) \conj{(T^\ast K_z)(w)}\,\dif\mu(w).
\]
Notice that
\[
\conj{(T^\ast K_z)(w)}
= \conj{\langle T^\ast K_z, K_w\rangle}
= \conj{\langle K_z, TK_w\rangle}
= \langle TK_w, K_z \rangle
= (TK_w)(z).
\]
Defining $K_T$ by~\eqref{eq:integral_kernel_associated_to_operator_in_RKHS},
we get~\eqref{eq:integral_representation_of_bounded_operator}.
Moreover,
\begin{equation}
\label{eq:conj_KT_zw}
\conj{K_T(z,w)}
= \conj{(T K_w)(z)}
= (T^\ast K_z)(w).
\end{equation}
Identities~\eqref{eq:integral_kernel_associated_to_operator_in_RKHS} and~\eqref{eq:conj_KT_zw} can be rewritten in the following short form:
\begin{equation}
\label{eq:KT1_and_conjKT2_as_images}
K_T(\cdot, w) = T K_w,
\qquad
\conj{K_T(z, \cdot)}
= T^\ast K_z.
\end{equation}
Since $K_z,K_w$ belong to $H$
and $T,T^\ast$ act from $H$ to $H$,
we get
\eqref{eq:KT1_in_H} and~\eqref{eq:KT2_in_H}.
\end{proof}

\begin{prop}
\label{prop:integral_operator_uniqueness}
Let $T$ and $K_T$ be as in Proposition~\ref{prop:bounded_operators_as_integral_operators}.
Suppose that
$M\colon X^2\to\bC$ is a function such that
\begin{equation}
\label{eq:M_second_condition}
\conj{M(z, \cdot)}\in H\qquad(z\in X)
\end{equation}
and
\begin{equation}
\label{eq:T_via_integral_with_another_function_M}
(T f)(z)
=
\int_X M(z, w) f(w)\,\dif\mu(w)
\qquad(f\in H,\ z\in X).
\end{equation}
Then, $K_T = M$.
\end{prop}

\begin{proof}
Let $z,a\in X$.
We apply the reproducing property for
$\conj{M(z, \cdot)}$
and then~\eqref{eq:T_via_integral_with_another_function_M} for $f=K_a$:
\begin{align*}
M(z,a)
&=
\conj{\conj{M(z,a)}}
=
\conj{\langle \conj{M(z,\cdot)}, K_a\rangle}
=
\langle K_a, \conj{M(z,\cdot)}\rangle
\\
&=
\int_X M(z,w) K_a(w)\,\dif\mu(w)
=
(T K_a)(z)
=
K_T(z, a).
\end{align*}
Thus, $M=K_T$.
\end{proof}

\begin{cor}
\label{cor:uniqueness_equivalent_condition}
Let $T$ and $K_T$ be as in Proposition~\ref{prop:bounded_operators_as_integral_operators},
and $M\colon X^2\to\bC$ satisfies~\eqref{eq:T_via_integral_with_another_function_M}.
Then, \eqref{eq:M_second_condition} is equivalent to $K_T=M$.
\end{cor}

\begin{proof}
It follows from property~\eqref{eq:KT2_in_H}
and Proposition~\ref{prop:integral_operator_uniqueness}.
\end{proof}

\begin{example}
\label{example:M_without_second_condition}
What can happen if $M$ does not satisfy~\eqref{eq:M_second_condition}?
Let us consider an example.
We denote the semi-Hilbert space of all square-integrable functions with respect to $\mu$ by $\cL^2(X)$
and the set of all measurable functions that are zero almost everywhere by $\cZ(X)$.
With this notation,
$L^2(X)$ is the quotient space $\cL^2(X)/\cZ(X)$.
We have to work with $\cL^2(X)$ instead of $L^2(X)$
because $M$ is a function, not a class of equivalent functions.
We assume that $\{f+\cZ(X)\colon\ f\in H\}$ does not coincide with $L^2(X)$.
Let $h$ belong to $H\setminus\{0\}$
and let $g$ be a nontrivial element of the orthogonal complement of $H$.
The latter phrase means that $g\in\cL^2(X)$,
$g\notin\cZ(X)$, and for every $f$ in $H$,
\[
\int_X f\,\conj{g}\,\dif\mu=0.
\]
We define $M\colon X^2\to\bC$ by
\[
M(z,w) \eqdef h(z)\,\conj{g(w)}\qquad(z,w\in X)
\]
and consider the integral operator $T$ defined by~\eqref{eq:T_via_integral_with_another_function_M}.
The assumption $g\in\cL^2(X)$ assures that the integral in the right-hand side of~\eqref{prop:integral_operator_uniqueness} exists in the Lebesgue sense.
Furthermore, for every $f$ in $H$ and every $z$ in $X$,
\[
(Tf)(z)
=\int_X h(z)\,\conj{g(w)}\,f(w)\,\dif\mu(w)
=h(z) \int_X f\,\conj{g}\,\dif\mu
=0.
\]
Thus, $T$ is the zero operator and $K_T=0$,
although $M$ is nonzero.
We conclude that condition~\eqref{eq:M_second_condition}
is not necessary to define a bounded linear integral operator acting in $H$,
but it is crucial to recover $M$ from $T$ by $M=K_T$.
In this example, $M$ satisfies
\begin{equation}
\label{eq:M_first_condition}
M(\cdot,w)\in H\qquad(w\in X),
\end{equation}
but this property does not help to obtain $M=K_T$.
\end{example}

What conditions on $M$ are necessary and sufficient to define a bounded linear operator $T$ using~\eqref{eq:T_via_integral_with_another_function_M}?
As we know, this is an open question, even in the case of Toeplitz operators, where $M(z,w)=b(w)K_w(z)$
with some measurable (possible unbounded) function $b$.

\section{Diagonalization of translation-invariant operators}
\label{sec:translation_invariant} 

In this section, we review some constructions and facts from~\cite{HerreraMaximenkoRamos2022}.
We freely use basic facts
about the Fourier transform and convolutions
over locally compact abelian groups;
see, e.g., \cite[Chapters~I--IV]{Folland2016}.

Let $G$ be a locally compact abelian group with a Haar measure $\nu$
and $(Y,\la)$ be a measure space.
We assume that $G$ is metrizable and $\sigma$-compact,
$Y$ is $\sigma$-finite,
and  $L^2(G,\nu)$ and $L^2(Y,\la)$ are separable.
Let $H$ be an RKHS on $G\times Y$ such that $H$ is naturally embedded into $L^2(G\times Y)$, i.e., the inner product in $H$ is given by
\[
\langle f,g\rangle
=\int_{G\times Y}f\,\overline{g}\,
\dif\,(\nu\otimes\la).
\]
Let $(K_{x,y})_{(x,y)\in G\times Y}$ be the reproducing kernel of $H$.
We suppose that $H$ is invariant under the horizontal translations.
Equivalently, $K$ satisfies
\begin{equation}
\label{eq:K_is_invariant}
K_{x,y}(u,v)=K_{0,y}(u-x,v)\qquad(u,x\in G,\ v,y\in Y).
\end{equation}
We require a technical assumption on $K$ which helps to justify the interchange of some integrals:
\begin{equation}
\label{eq:K_tecnhical_assumption}
\forall y\in Y\qquad
\sup_{v\in Y}
\int_G |K_{0,y}(u,v)|\,\dif\nu(u)
<+\infty.
\end{equation}
We denote by $\cU(H)$ the group of all unitary operators acting in $H$.
Define $\rho\colon G\to\cU(H)$ by
\begin{align}\label{eq:rho_H(a)}
(\rho(a)f)(x,y)\eqdef f(x-a,y)\qquad(a,x\in G,\ y\in Y,\ f\in H).
\end{align}
It is easy to see that $\rho$ is a unitary representation of $G$ on $H$.
Our main object of study is the centralizer of $\rho$ which we denote by $\cC(\rho)$:
\begin{equation}
\label{eq:centralizer_rho_H}
\cC(\rho)
\eqdef
\bigl\{
S\in\cB(H)\colon\quad
\forall a\in G\quad
S\rho(a)=\rho(a)S
\bigr\}.
\end{equation}
It is well known
(see, e.g., Folland~\cite[Section 3.1, page~75]{Folland2016}) that the centralizer (commutant) of a unitary representation is a von Neumann algebra.
Indeed, it is the commutant of the self-adjoint subset $\{\rho(a)\colon\ a\in G\}$ of $\cB(H)$.

Following the convention from Remark~\ref{rem:dot_notation_for_functions_with_fixed_arguments},
we use the points $\cdot$ or $\bullet$
to indicate the ``active variables'',
when some arguments of a function are fixed.
For example, $L_{\bullet,y}(v)$ below is the function $\hG\to\bC$
acting by
$\xi\mapsto L_{\xi,y}(v)$,
and $K_{0,y}(\cdot,v)$ is the function $G\to\bC$
acting by $u\mapsto K_{0,y}(u,v)$.

Let $\hG$ be the dual group of $G$.
We denote by $\hnu$ the Haar measure on $\hG$ such that the Fourier--Plancherel transform
$F\colon L^2(G,\nu)\to L^2(\hG,\hnu)$ is isometric.
We denote by $L_{\bullet,y}(v)$ the Fourier transform of $K_{0,y}(\cdot,v)$:
\begin{equation}
\label{eq:L_def}
L_{\xi,y}(v)
\eqdef
\int_G K_{0,y}(u,v)\,\overline{\xi(u)}\,\dif\nu(u)
\qquad(\xi\in\hG,\ y,v\in Y).
\end{equation}
Let $\Om$ be the set of all
``nontrivial frequencies'':
\[
\Om\eqdef\bigl\{\xi\in\hG\colon\ 
\exists y\in Y\quad L_{\xi,y}(y)>0\bigr\}.
\]
The following simple fact was not noticed in~\cite{HerreraMaximenkoRamos2022}.

\begin{prop}
\label{prop:Om_is_open}
$\Om$ is an open subset of $\hG$.
\end{prop}

\begin{proof}
Let $\eta\in\Om$.
We choose $y$ in $Y$ such that $L_{\eta,y}(y)>0$.
Thanks to~\eqref{eq:K_tecnhical_assumption},
$K_{0,y}(\cdot, y)\in L^1(G)$.
Thus, $L_{\bullet,y}(y)$ is a continuous function.
Therefore, there is an open neighborhood $A$ of $\eta$
such that $L_{\xi,y}(y)>0$ for each $\xi$ in $A$.
\end{proof}

As it was shown in~\cite{HerreraMaximenkoRamos2022}, for each $\xi$ in $\Om$, $(L_{\xi,y})_{y\in Y}$ is
the reproducing kernel of a certain RKHS $\hH_\xi$ which is naturally embedded into $L^2(Y)$, and the horizontal Fourier transform of $H$ decomposes into the direct integral of these ``fibers'':
\[
(F\otimes I)H
=\int_{\Om}^{\oplus} \hH_\xi\,\dif\hnu(\xi).
\]
We suppose there exists a measurable function family $(q_\xi)_{\xi\in\Om}$ such that
\begin{equation}
\label{eq:L_decomposition}
L_{\xi,y}(v)=\overline{q_\xi(y)}q_\xi(v)
\qquad(\xi\in\Om,\ v,y\in Y).
\end{equation}
This means that for every $\xi$ in $\Om$, the space $\hH_\xi$
is one-dimensional:
$\hH_\xi=\bC q_\xi$.
Moreover, $q_\xi$ is normalized.  Let $\hH\eqdef(F\otimes I)H$
and
$\FourierH\colon H\to\hH$ be the compression of $F\otimes I$.  Define $N\colon\hH\to L^2(\Om)$ by
\[
(N g)(\xi) \eqdef
\langle g(\xi,\cdot), q_\xi\rangle_{L^2(Y)}
=
\int_Y g(\xi,v) \overline{q_\xi(v)}\,\dif{}\la(v).
\]
Then $N$ is an isometric isomorphism, and its adjoint acts by
\[
(N^\ast h)(\xi,v)
= h(\xi) q_\xi(v).
\]
Define $R\colon H\to L^2(\Om)$
as the composition
\[
R \eqdef N \FourierH.
\]
More explicitly,
\[
(R f)(\xi)
=
\langle ((F\otimes I)f)(\xi,\cdot),q_\xi\rangle_{L^2(Y)}
=\int_Y ((F\otimes I)f)(\xi,v)\, \overline{q_\xi(v)}\,\dif\la(v).
\]
For every $f$ in $H$ such that $f\in L^1(G\times Y)$,
\begin{equation}
\label{eq:R_integral_representation}
(R f)(\xi)
=
\int_{G\times Y}
f(u,v)\,
\overline{\xi(u)q_\xi(v)}\,
\dif\nu(u)\,
\dif\la(v).
\end{equation}
Then $R$ is an isometric isomorphism of Hilbert spaces.
Its adjoint acts by
$R^\ast=\FourierH^\ast N^\ast$.

Given $b$ in $L^\infty(\Om)$, we denote by $M_b$ the corresponding multiplication operator acting in $L^2(\Om)$.
Moreover, for every $b$ in $L^\infty(\Om)$,
we define $V_b\in\cB(H)$ by
\begin{equation}
\label{eq:Vb_def}
V_b \eqdef R^\ast M_b R.
\end{equation}
It is easy to see that $V_b\in\cC(\rho)$.
Figure~\ref{fig:operators} shows the main objects of this construction.

\tikzset{spacenode/.style
={rounded corners, text centered,
	text width=10ex, minimum size = 4ex,
	draw, fill = lightgreen,
	outer sep = 0.4ex}}

\tikzset{widespacenode/.style
={rounded corners, text centered,
	text width=16ex, minimum size = 4ex,
	draw, fill = lightgreen,
	outer sep = 0.4ex}}

\tikzset{myedge/.style
={darkgreen, thick, -Stealth}}

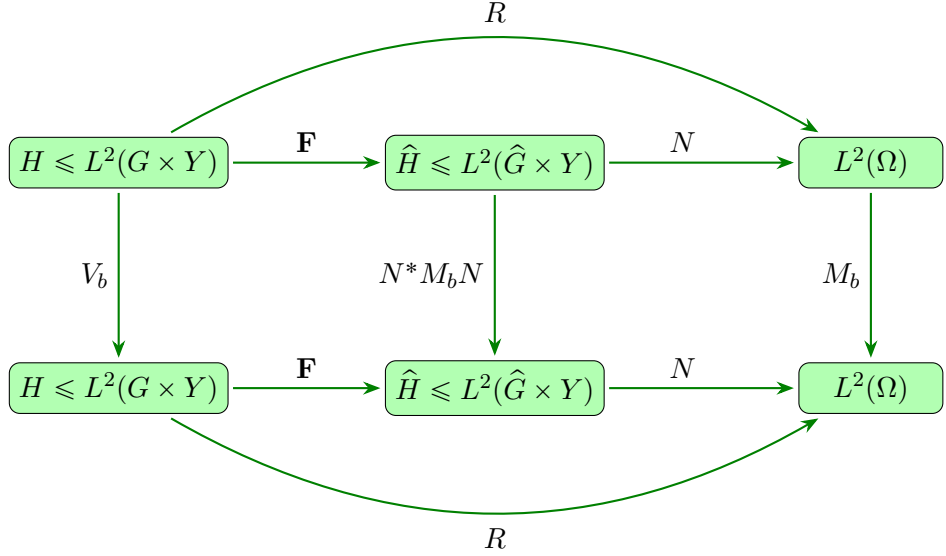
\begin{figure}[hbt]
\centering
\begin{tikzpicture}
	\node [widespacenode] (H1)
	at (-5, 0) {$H\le L^2(G\times Y)$};
	\node [widespacenode] (H2)
	at (-5, -3) {$H\le L^2(G\times Y)$};
	\node [widespacenode] (hH1)
	at (0, 0) {$\hH\le L^2(\hG\times Y)$};
	\node [widespacenode] (hH2)
	at (0, -3) {$\hH\le L^2(\hG\times Y)$};
	\node [spacenode] (L2Om1)
	at (5, 0) {$L^2(\Om)$};
	\node [spacenode] (L2Om2)
	at (5, -3) {$L^2(\Om)$};
	% isometric isomorphisms
	\draw [myedge]
	(H1) edge (hH1);
	\node at (-2.5, 0) [above]
	{$\FourierH$};
	\draw [myedge]
	(H2) edge (hH2);
	\node at (-2.5, -3) [above]
	{$\FourierH$};
	\draw [myedge]
	(hH1) edge (L2Om1);
	\draw [myedge]
	(hH2) edge (L2Om2);
	\node at (2.5, 0) [above]
	{$N$};
	\node at (2.5, -3) [above]
	{$N$};
	\draw [myedge, bend left]
	(H1) edge (L2Om1);
	\node at (0, 2)
	{$R$};
	\draw [myedge, bend right]
	(H2) edge (L2Om2);
	\node at (0, -5)
	{$R$};
	% operators
	\draw [myedge] (H1) edge (H2);
	\node at (-5, -1.5) [left]
	{$V_b$};
	\draw [myedge]
	(hH1) edge (hH2);
	\node at (0, -1.5) [left]
	{$N^\ast M_b N$};
	\draw [myedge]
	(L2Om1) -- (L2Om2);
	\node at (5, -1.5) [left]
	{$M_b$};
\end{tikzpicture}
\caption{
	\label{fig:operators}
	$V_b$ and $M_b$.
}
\end{figure}

Define
$\La\colon L^\infty(\Om)\to\cC(\rho)$
by
\[
\La(b) \eqdef V_b.
\]

We use the following terminology
(see, e.g., Blackadar~\cite[Section III.1.8]{Blackadar2006}
or a slightly different approach in Sakai~\cite[Sections~1.1 and 1.16]{Sakai1971}).
A \emph{von Neumann algebra} is a self-adjoint subalgebra of $\cB(H)$
(where $H$ is a Hilbert space) containing the identity operator and closed in the weak operator topology.
An \emph{isomorphism} of C*-algebras is a *-isomorphism.
Every isomorphism of C*-algebras is isometric.
A \emph{W*-algebra}
is a C*-algebra which is isomorphic to some von Neumann algebra.
In particular, we always suppose W*-algebras to be unital.

It is well known 
(see, e.g., Sunder~\cite[Exercise~(0.4.5)]{Sunder1987} or Murphy~\cite[Examples~2.5.1 and 4.1.2]{Murphy1990})
that $L^\infty(\Om)$ is isomorphic to the von Neumann algebra of the corresponding multiplication operators;
thus, $L^\infty(\Om)$ is a W*-algebra.

The following theorem was proved in~\cite{HerreraMaximenkoRamos2022}.

\begin{thm}
\label{thm:centralizer_through_Vb}
$\La$ is an isomorphism
of W*-algebras $L^\infty(\Om)$ and $\cC(\rho)$.
In particular,
\begin{equation}
\label{eq:centralizer_via_R}
\cC(\rho)
=\bigl\{V_b\colon\ b\in L^\infty(\Om)\bigr\}.
\end{equation}
\end{thm}

In other words,
this theorem means that
$R$ diagonalizes each operator belonging to $\cC(\rho)$,
and $R\cC(\rho)R^\ast$
is the algebra of all multiplication operators acting in $L^2(\Om)$.
As a consequence, $\cC(\rho)$ is commutative.

Given $a$ in $G$,
we define $E_a\colon\Om\to\bC$ by
\[
E_a(\xi) \eqdef \xi(a).
\]
In the following proposition,
we compute the image of $K_{x,y}$
with respect to $\FourierH$ and $R$;
see Figure~\ref{fig:image_of_K}.
This is a simple generalization
of~\cite[Proposition~7.4]{HerreraMaximenkoRamos2022}.

\begin{prop}
\label{prop:RK}
Let $x\in G$, $y\in Y$.
Then
$R K_{x,y}
=E_{-x}\,\overline{q_{\bullet}(y)}$,
i.e., for every $\xi$ in $\Om$,
\begin{equation}
\label{eq:RK}
(R K_{x,y})(\xi)
=\overline{\xi(x) q_\xi(y)}.
\end{equation}
Moreover, $q_{\bullet}(y)\in L^2(\Om)$,
and
\begin{equation}
\label{eq:norm_q_with_fixed_y}
\int_{\Om} L_{\xi,y}(y)\,\dif\hnu(\xi)
=\int_{\Om} |q_\xi(y)|^2\,\dif\hnu(\xi)
=K_{0,y}(0,y).
\end{equation}
\end{prop}

\begin{proof}
By~\eqref{eq:K_is_invariant}
and~\eqref{eq:L_def},
\begin{align*}
(\FourierH K_{x,y})(\xi,v)
&=
\int_G K_{0,y}(u-x,v)\,\overline{\xi(u)}\,\dif\nu(u)
=
\overline{\xi(x)}\,
\int_G K_{0,y}(t,v)\,\overline{\xi(t)}\,\dif\nu(t)
\\
&=
\overline{\xi(x)}\,
L_{\xi,y}(v)
=
\overline{\xi(x)q_\xi(y)}\,
q_\xi(v).
\end{align*}
So,
\[
(\FourierH K_{x,y})(\xi, \cdot)
=
\overline{\xi(x)q_\xi(y)}\,q_\xi(\cdot).
\]
By definitions of $R$ and $N$,
this implies that
\[
(R K_{x,y})(\xi)
= (N \FourierH K_{x,y})(\xi)
= \overline{\xi(x) q_\xi(y)}.
\]
Finally, by the isometric property of $R$,
\[
\int_\Om |q_\xi(y)|^2\,\dif\hnu(\xi)
=\bigl\|q_{\bullet}(y)\bigr\|_{L^2(\Om)}^2
=\bigl\|K_{x,y}\bigr\|_H^2
=K_{x,y}(x,y)
=K_{0,y}(0,y).
\qedhere
\]
\end{proof}

\begin{figure}[hbt]
\centering
\begin{tikzpicture}
	\node [widespacenode, text width=18ex] (H1)
	at (-5, 0) {$K_{x,y}\in H$};
	\node [widespacenode, text width=20ex] (hH1)
	at (0, 0) {$E_{-x} L_{\bullet,y}(\cdot)\in \hH$};
	\node [spacenode, text width=20ex] (L2Om1)
	at (5, 0) {$E_{-x}\,\overline{q_{\bullet}(y)}\in L^2(\Om)$};
	% isometric isomorphisms
	\draw [myedge]
	(H1) edge (hH1);
	\node at (-2.5, 0) [above]
	{$\FourierH$};
	\draw [myedge]
	(hH1) edge (L2Om1);
	\node at (2.5, 0) [above]
	{$N$};
	\draw [myedge, bend left]
	(H1) edge (L2Om1);
	\node at (0, 2)
	{$R$};
\end{tikzpicture}
\caption{
	\label{fig:image_of_K}
 Image of $K_{x,y}$
 with respect to $\FourierH$
 and $R$.
}
\end{figure}
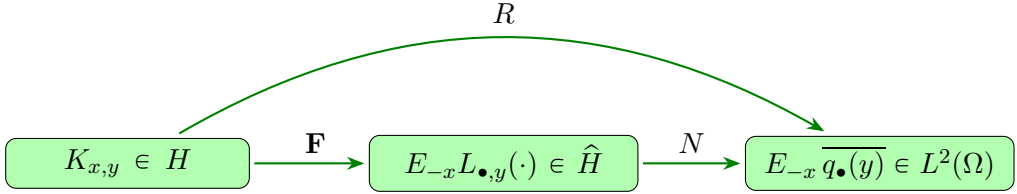

The following technical proposition
is used in Proposition~\ref{prop:Lambda_inverse}.

\begin{prop}
\label{prop:partition_of_Omega}
There exists a finite or countable subset $Y_0$ of $Y$
and a measurable partition $(\Om_y)_{y\in Y_0}$ of $\Om$
such that for every $\xi$ in $\Om_y$,
$L_{\xi,y}(y)>0$.
\end{prop}

\begin{proof}
The technical conditions on $G$ imply that $\hG$ is a metrizable separable space.
Therefore, $\hG$ and $\Om$ are Lindel\"{o}f spaces.
Let $P$ be the set of all pairs $(y,A)$, where
$y\in Y$, $A$ is an open subset of $\Om$,
and $L_{\xi,y}(y)>0$ for every $\xi$ in $A$.
Reasoning as in the proof of Proposition~\ref{prop:Om_is_open},
we see that the collection $\{A\colon\ (y,A)\in P\}$
is an open cover of $\Om$.
Therefore, there is a finite or countable subcover of $\Om$.
In other words,
there exists a set $J\subseteq\bN$
($J=\bN$ or $J=\{1,\ldots,m\}$, $m\in\bN$)
and a family $((y_j, A_j))_{j\in J}$
such that $(y_j,A_j)\in P$ for each $j$ in $J$
and $\bigcup_{j\in J}A_j=\Om$.
For each $j$ in $J$, we set
$B_j\eqdef A_j\setminus\left(\bigcup_{k<j}A_k\right)$.
Then, the sets $B_j$ are pairwise disjoint,
and $\bigcup_{j\in J}B_j=\Om$.
Let $Y_0\eqdef\{y_j\colon\ j\in J,\ B_j\ne\emptyset\}$.
For each $y$ in $Y_0$, we set $\Om_y\eqdef\cup\{B_j\colon\ y_j=y\}$.
Thus, we obtain a finite or countable family $(\Om_y)_{y\in Y_0}$
whose union is $\Om$
and whose elements are measurable subsets of $\Om$.
\end{proof}

Now, we are ready to give an explicit formula for $\La^{-1}$.
The next proposition is a more precise version
of \cite[Corollary 7.8]{HerreraMaximenkoRamos2022}.
In the present version,
we emphasize that~\eqref{eq:Lambda_inverse}
is an equality of equivalence classes
(in other words, it can be treated as an equality almost everywhere),
and we explain how to define
the right-hand side of~\eqref{eq:Lambda_inverse}.

We recall a technical concept from measure theory.
Suppose that $f$ is an element of $L^2(\Om)$, i.e.,
$f$ is an equivalence class
consisting of some functions $g$ belonging to $\cL^2(\Om)$,
and suppose that $Z$ is a measurable subset of $\Om$.
Then, the restriction $f|_Z$ is defined as the set of the restrictions of all representants of $f$:
\[
f|_Z \eqdef \{g|_Z\colon\ g\in f\}.
\]

\begin{prop}
\label{prop:Lambda_inverse}
Let $S\in\cC(\rho)$.
Then, the following equality holds in $L^\infty(\Om)$:
\begin{equation}
\label{eq:Lambda_inverse}
(\La^{-1}(S))(\xi)
=\frac{(RSK_{0,y})(\xi)}{\conj{q_\xi(y)}}
\qquad(\xi\in\Om),
\end{equation}
where $y\in Y_0$ such that $\xi\in\Om_y$.
\end{prop}

\begin{proof}
First, let us explain the definition of the right-hand side of~\eqref{eq:Lambda_inverse}.
For each $y$ in $Y_0$,
$\xi\mapsto \conj{q_\xi(y)}$ is a nonvanishing measurable function defined on $\Om_y$,
and the restriction of $RSK_{0,y}$ to $\Om_y$ is an element of $L^2(\Om_y)$.
Therefore, the quotient
\[
b_y(\xi) \eqdef \frac{(RSK_{0,y})(\xi)}{\conj{q_\xi(y)}}
\qquad(\xi\in\Om_y)
\]
is an equivalence class of measurable functions on $\Om_y$.
We define $b$ as an equivalence class of measurable functions on $\Om$
such that for every $y$ in $Y_0$, $b|_{\Om_y}=b_y$.
Since $(\Om_y)_{y\in Y_0}$ is a finite or countable measurable partition of $\Om$, this definition makes sense.

By Theorem~\ref{thm:centralizer_through_Vb},
there exists $c$ in $L^\infty(\Om)$ such that $S=V_c=R^\ast M_c R$.
Let us prove that $b=c$.
For every $y$ in $Y_0$,
by Proposition~\ref{prop:RK},
$RK_{0,y}=\conj{q_{\bullet}(y)}$.
Therefore, we obtain the following equality in $L^2(\Om_y)$:
\[
c(\xi)\conj{q_\xi(y)}
=(M_c\, \conj{q_{\bullet}(y)})(\xi)
=(RS K_{0,y})(\xi)
=b_y(\xi)\,\conj{q_\xi(y)}
=b(\xi)\,\conj{q_\xi(y)}.
\]
We conclude that the restrictions of $b$ and $c$ to $\Om_y$ coincide on each $y$ in $Y_0$, and $b=c$.

If we work with functions instead of equivalence classes,
then the latter reasoning takes the following form.
For each $y$ in $Y_0$,
we get $\hnu(\{\xi\in\Om_y\colon\ b(\xi)\ne c(\xi)\})=0$.
The union of these sets has measure $0$
because $Y_0$ is finite or countable.
\end{proof}

\section{Integral form of translation-invariant operators}\label{sec:integral_representation}

In this section, our aim is to write every operator $V_b$, where $b\in L^\infty(\Om)$,
as an integral operator.
First, we apply Proposition~\ref{prop:bounded_operators_as_integral_operators} to our situation where $X=G\times Y$ and $T\in\cC(\rho)$.

\begin{prop}
\label{prop:initial_integral_form_of_bouned_invariant_operator}
Let $T\in \cC(\rho)$.
Define $K_T\colon G\times Y\times G\times Y\to\bC$ by
\begin{equation}
\label{eq:KT_four_arguments_via_TK}
K_T(x,y,u,v)
\eqdef (T K_{u,v})(x,y)
\qquad(x,u\in G,\ y,v\in Y).
\end{equation}
Then, for every $f$ in $H$
and every $(x,y)$ in $G\times Y$,
\begin{equation}
\label{eq:initial_integral_form_of_bouned_invariant_operator}
(T f)(x,y)
=
\int_{G\times Y}
K_T(x,y,u,v)
f(u,v)
\,\dif\nu(u)\,\dif\la(v).
\end{equation}
Moreover, for every $u,x$ in $G$
and $v,y$ in $Y$,
\begin{equation}
\label{eq:KT_invariant_property}
K_T(x,y,u,v) = K_T(x-u, y, 0, v)
\end{equation}
and
\begin{equation}
\label{eq:KT_1_2_belongs_to_H}
K_T(\cdot, \cdot, 0, v)\in H,\qquad
\overline{K_T(0, y, \cdot, \cdot)}\in H.
\end{equation}
\end{prop}

\begin{proof}
\eqref{eq:initial_integral_form_of_bouned_invariant_operator}
and~\eqref{eq:KT_1_2_belongs_to_H}
follow directly from Proposition~\ref{prop:bounded_operators_as_integral_operators}.
By~\eqref{eq:K_is_invariant},
\[
\rho(a) K_{0,b} = K_{a,b}
\qquad(a\in G,\ b\in Y).
\]
The assumption $T\in\cC(\rho)$ implies that
$\rho(u) T = T \rho(u)$.
So,
\begin{align*}
K_T(x,y,u,v)
&= (T K_{u,v})(x,y)
= (T \rho(u) K_{0,v}(x,y))
\\
&= (\rho(u) T K_{0,v})(x,y)
= (T K_{0,v})(x-u, y)
= K_T(x-u, y, 0, v).
\qedhere
\end{align*}
\end{proof}

\begin{defn}
\label{def:phi_b}
Given $b$ in $L^\infty(\Om)$, we define
$\phi_b\colon G\times Y\times Y \to \bC$ by
\begin{equation}
\label{eq:integral_kernel_from_spectral_function}
\phi_b(x,y,v)
\eqdef
\int_{\Om} \xi(x)\, L_{\xi,v}(y)\, b(\xi)\, \dif\hnu(\xi).
\end{equation}
\end{defn}

Equivalently, due to our assumption~\eqref{eq:L_decomposition},
\begin{equation}
\label{eq:integral_kernel_from_spectral_function_via_q}
\phi_b(x,y,v)
=
\int_{\Om} \xi(x)\, q_\xi(y)\, \conj{q_\xi(v)}\, b(\xi) \dif\hnu(\xi).
\end{equation}

\begin{rem}
\label{rem:integral_in_def_phi_exists_in_Lebesgue_sense}
The integral in~\eqref{eq:integral_kernel_from_spectral_function}
exists in the Lebesgue sense.
Since $L_{\xi,\bullet}(\cdot)$ is the reproducing kernel of $\hH_\xi$,
it satisfies the Schwarz inequality:
\begin{equation}
\label{eq:Schwarz_inequality_for_L}
|L_{\xi,v}(y)|
\le \sqrt{L_{\xi,v}(v)}\,\sqrt{L_{\xi,y}(y)}.
\end{equation}
Applying~\eqref{eq:Schwarz_inequality_for_L},
the usual Schwarz (or H\"{o}lder) inequality,
and~\eqref{eq:norm_q_with_fixed_y}, we get
\begin{align*}
\int_\Om |\xi(x)\,L_{\xi,v}(y)\,b(\xi)|\,\dif\hnu(\xi)
&\le
\|b\|_\infty
\left(\int_\Om L_{\xi,y}(y)\,\dif\hnu(\xi)\right)^{1/2}
\left(\int_\Om L_{\xi,v}(v)\,\dif\hnu(\xi)\right)^{1/2}
\\[0.5ex]
&=
\|b\|_\infty\, \|K_{0,y}\|\,\|K_{0,v}\|.
\end{align*}
\end{rem}

\begin{lem}
\label{lem:VK}
Let $b\in L^\infty(\Om)$.
Then, 
for every $u$ in $G$
and every $v,y$ in $Y$,
\begin{equation}
\label{eq:phi_via_VK}
(V_b K_{0,y})(u,v)
=\phi_b(u,v,y).
\end{equation}
Moreover,
for every $u,x$ in $G$
and every $v,y$ in $Y$,
\begin{equation}
\label{eq:V_b_K}
(V_b K_{x,y})(u,v)
=\phi_b(u-x,v,y)
\end{equation}
and
\begin{equation}
\label{eq:V_conjb_K}
(V_{\conj{b}}K_{x,y})(u,v)
=\conj{\phi_b(x-u,y,v)}.
\end{equation}
\end{lem}

\begin{proof}
Recall that
\[
V_b
=R^\ast M_b R
=\FourierH^\ast N^\ast M_b R.
\]
First, we compute
$N^\ast M_b R K_{x,y}$
using
Proposition~\ref{prop:RK}:
\[
(N^\ast M_b R K_{x,y})(\xi,v)
=
q_\xi(v) b(\xi)
(R K_{x,y})(\xi,v)
=
\conj{\xi(x)} 
b(\xi)
q_\xi(v)\,\conj{q_\xi(y)}
=
\conj{\xi(x)}
b(\xi)
L_{\xi,y}(v).
\]
Now, we apply
$\FourierH^\ast$
and compare the obtained expression
with the definition of $\phi_b$:
\begin{align*}
(V_b K_{x,y})(u,v)
&=
\int_{\hG}
\xi(u)
\conj{\xi(x)}
b(\xi)
L_{\xi,y}(v)\,
\dif\hnu(\xi)
\\[1ex]
&=
\int_{\hG}
\xi(u-x)
b(\xi)
L_{\xi,y}(v)\,
\dif\hnu(\xi)
= \phi_b(u-x,v,y).
\end{align*}
Thereby, we have proved~\eqref{eq:V_b_K} and its particular case~\eqref{eq:phi_via_VK}.
Similarly,
\begin{align*}
(V_{\conj{b}}K_{x,y})(u,v)
&=
\int_{\hG}
\xi(u-x)
\conj{b(\xi)}
L_{\xi,y}(v)\,
\dif\hnu(\xi)
\\[1ex]
&=
\conj{\int_{\hG} \xi(x-u) b(\xi) L_{\xi,v}(y)\,\dif\hnu(\xi)}
= \conj{\phi_b(x-u,y,v)}.
\qedhere
\end{align*}
\end{proof}

\begin{thm}
\label{thm:V_integral_representation}
Let $b\in L^\infty(\Om)$.
Then for every $f$ in $H$,
$x$ in $G$ and $y$ in $Y$,
\begin{equation}
\label{eq:V_integral_representation}
(V_b f)(x,y)
=
\int_{G\times Y}
f(u,v) \phi_b(x-u,y,v)
\,\dif\nu(u)\,\dif\la(v).
\end{equation}
\end{thm}

\begin{proof}
By~\eqref{eq:centralizer_via_R},
we know that $V_b\in\cC(\rho)$.
We apply Proposition~\ref{prop:initial_integral_form_of_bouned_invariant_operator}
to $T\eqdef V_b$.
In this case,
by~\eqref{eq:phi_via_VK},
\[
K_T(x,y,0,v)
=(T K_{0,v})(x,y)
=(V_b K_{0,v})(x,y)
=\phi_b(x,y,v).
\]
Finally, by~\eqref{eq:KT_invariant_property},
\[
K_T(x,y,u,v)
=K_T(x-u,y,0,v)
=\phi_b(x-u,y,v).
\]
So, \eqref{eq:initial_integral_form_of_bouned_invariant_operator} transforms to~\eqref{eq:V_integral_representation}.
\end{proof}

\section{The class of translation-invariant integral kernels}
\label{sec:cA}

The following definition is motivated by Proposition~\ref{prop:initial_integral_form_of_bouned_invariant_operator}.

\begin{defn}
Let $\cA_0$ be the set of all functions
$\psi\colon G\times Y\times Y \to \bC$
satisfying the following two conditions:
\begin{align}
\label{eq:psi_cond_1}
\psi(\cdot,\cdot,v)\in H
\qquad(v\in H),
\\[1ex]
\label{eq:psi_cond_2}
\conj{\psi(-\cdot,y,\cdot)}\in H
\qquad (y\in Y).
\end{align}
\end{defn}

Since $H$ is translation invariant,
the following properties hold for every $\psi$ in $\cK_0$:
\begin{align}
\label{eq:psi_daughters1_belong_to_H}
&\psi(\cdot-u,\cdot,v)\in H
\qquad (u\in G,\ v\in Y),
\\[1ex]
\label{eq:psi_daughters2_belong_to_H}
&\conj{\psi(x-\cdot,y,\cdot)} \in H
\qquad (x\in G,\ y\in Y).
\end{align}

\begin{defn}
Let
$\cD \eqdef
\linspan
\bigl\{K_{p,q}\colon\ p\in G,\, q\in Y\bigr\}$.
Given $\psi$ in $\cA_0$,
define $\tS_\psi\colon\cD\to H$ by
\begin{equation}
\label{eq:widetilde_S_def}
(\tS_\psi f)(x,y)
\eqdef \int_{G\times Y} f(u,v)\,
\psi(x-u,y,v)\, \dif\nu(u)\, \dif\la(v).
\end{equation}
\end{defn}

The following lemma shows
that $\tS_\psi$
is well defined on $\cD$.

\begin{lem}
\label{lem:tildeS_K_eq_psi}
Let $\psi\in\cA_0$, $p\in G$, and $q\in Y$.
Then $\tS_\psi K_{p,q}\in H$, and
\[
(\tS_\psi K_{p,q})(x,y)
=\psi(x-p,y,q)
\qquad(x\in G,\ y\in Y).
\]
\end{lem}

\begin{proof}
Let $x\in G$ and $y\in Y$.
Define $g\colon G\times Y\to\bC$ by
\[
g(u,v) \eqdef \conj{\psi(x-u,y,v)}.
\]
By~\eqref{eq:psi_daughters2_belong_to_H},
$g\in H$.
So, by the reproducing property,
\begin{align*}
	(\tS_\psi K_{p,q})(x,y) 
	&=
    \int_{G\times Y} K_{p,q}(u,v)\,
    \psi(x-u,y,v)\, \dif\nu(u)\, \dif\la(v)
    \\
    &=
    \langle K_{p,q}, g\rangle
    =
    \conj{\langle g, K_{p,q}\rangle}
    =
    \conj{g(p,q)}
    =
    \psi(x-p,y,q).
\end{align*}
As $\psi \in \cA_0$,
it follows by~\eqref{eq:psi_daughters1_belong_to_H} 
that
$(\tS_\psi K_{p,q})(\cdot,\cdot)
=\psi(\cdot - p, \cdot,q) \in H$.
\end{proof}

\begin{rem}
As we see in the proof of Lemma~\ref{lem:tildeS_K_eq_psi},
condition~\eqref{eq:psi_cond_2}
allows us to write $\tS_\psi f$ as an inner product in $H$,
and
condition~\eqref{eq:psi_cond_1}
assures that $\tS_\psi f\in H$ for $f$ in $\cD$.
\end{rem}

\begin{rem}
\label{rem:def_ttS}
If $\psi\in\cA_0$,
then, by~\eqref{eq:psi_daughters2_belong_to_H},
for every $f$ in $H$,
$x\in G$, and $y$ in $Y$,
the integral in the right-hand side of~\eqref{eq:widetilde_S_def}
exists in the Lebesgue sense
and can be written as an inner product:
\begin{equation}
\label{eq:def_S_via_inner_product}
\int_{G\times Y}
f(u,v)\,
\psi(x-u,y,v)\, \dif\nu(u)\, \dif\la(v)
=
\langle f, \conj{\psi(x - \cdot, y, \cdot)} \rangle.
\end{equation}
The obtained expression depends on $x,y$,
and we denote it by
$(\ttS_\psi f)(x,y)$.
In other words, we have defined a linear operator
$\ttS_\psi\colon H\to\bC^X$.
In general, the condition $\psi\in\cA_0$
is not sufficient to guarantee that
all values of this operator belong to $H$.
\end{rem}

\begin{defn}[integral operators of a special form]
We denote by $\cA$ the set of all functions
$\psi$ belonging to $\cA_0$
such that the following integral operator
$S_\psi\colon H\to H$ is well-defined and bounded:
\begin{equation}
\label{eq:S_def}
	(S_\psi f)(x,y)
     \eqdef
     \int_{G\times Y} f(u,v)\, \psi(x-u,y,v)\, \dif\nu(u)\, \dif\la(v) \qquad (f\in H).
\end{equation}
\end{defn}

\begin{rem}
\label{rem:def_cA_detailed_form}
More explicitly,
the condition $\psi\in\cA$ means
the following system of three conditions:
$\psi\in\cA_0$,
\begin{equation}
\label{eq:Sf_belongs_to_H}
\forall f\in H\qquad \ttS_\psi f\in H,
\end{equation}
and
\begin{equation}
\sup_{\substack{f\in H\\\|f\|=1}}
\|\ttS_\psi f\|
< +\infty.
\end{equation}
\end{rem}

\begin{prop}
\label{prop:cA_is_cA0_with_boundedness}
$\psi\in\cA$ if and only if
$\psi\in\cA_0$
and the linear operator $\tS_\psi$ is bounded on $\cD$:
\begin{equation}
\label{eq:sup_norm_over_cD}
\sup_{\substack{f\in\cD\\\|f\|=1}}
\|\tS_\psi f\| < +\infty.
\end{equation}
\end{prop}

\begin{proof}
The $\Rightarrow$ part if clear,
by Remark~\ref{rem:def_cA_detailed_form}.
To prove the $\Leftarrow$ part,
we suppose that $\psi\in\cA_0$ and $\tS_\psi$ is bounded on $\cD$.
It is well known
(see Kreyszig~\cite[Theorem~2.7-11]{Kreyszig1978})
that every densely defined bounded linear operator,
acting from a normed space to a Banach space,
admits a unique bounded extension defined on the whole space.
Thus, there exists a unique $T$ in $\cB(H)$ such that
$Tf=\tS_\psi f$ for every $f$ in $\cD$.
Given $(x,y)$ in $G\times Y$,
we define $h_{x,y}\colon G\times Y\to\bC$ by
$h_{x,y}(u,v)\eqdef \conj{\psi(x-u,y,v)}$.
Using~\eqref{eq:def_S_via_inner_product},
we rewrite the equality $(Tf)(x,y)=(\widetilde{S}_\psi f)(x,y)$
in the form
\[
\langle Tf, K_{x,y}\rangle
= \langle f, h_{x,y}\rangle
\qquad(f\in\cD,\ x\in G,\ y\in Y).
\]
Since both sides 
of this equality depend continuously on $f$
and $\cD$ is dense in $H$,
the equality extends to all $f$ in $H$:
\[
(T f)(x,y)
= \langle Tf, K_{x,y}\rangle
= \langle f,h_{x,y}\rangle
= (\ttS_{\psi} f)(x,y)
\qquad(f\in H,\ x\in G,\ y\in Y).
\]
We conclude that $Tf=\widetilde{\widetilde{S}}_{\psi} f$
for every $f$ in $H$.
By Remark~\ref{rem:def_cA_detailed_form},
this means that $S_\psi$ is a well-defined bounded linear operator,
and $\psi\in\cA$.
\end{proof}

In general, $\cA$ can be a proper subset of $\cA_0$;
see Remark~\ref{rem:A0_is_not_A}.

Now, we can rewrite Theorem~\ref{thm:V_integral_representation}
in the following short form.

\begin{cor}
\label{cor:Vb_eq_Sphib}
Let $b\in L^\infty(\Om)$.
Then, $\phi_b\in\cA$ and $V_b=S_{\phi_b}$.
\end{cor}

\begin{proof}
1. Let us verify that $\phi_b\in\cA_0$.
Since $V_b$ and $V_{\conj{b}}$
are bounded operators acting in $H$,
by~\eqref{eq:phi_via_VK} we get
\[
\phi_b(\cdot, \cdot, y)
= V_b K_{0,y} \in H.
\]
Similarly, by~\eqref{eq:V_conjb_K},
\[
\overline{\phi_b(x-\cdot, y, \cdot)}
= V_{\conj{b}} K_{x,y}
\in H.
\]
2. By~\eqref{eq:V_integral_representation},
we get
$\ttS_{\phi_b} f = V_b f$
for each $f$ in $H$.
Since $V_b\in\cB(H)$,
we conclude that $\phi_b\in\cA$
and $S_{\phi_b}=V_b$.
\end{proof}

\begin{lem}
\label{lem:S_psi_in_centralizer}
Let $\psi\in\cA$.
Then,
$S_\psi \in \cC(\rho)$.
\end{lem}

\begin{proof}
Let $a\in G$,
$f\in H$, $x\in G$, $y\in Y$.
Then
\begin{align*}
	(\rho(a)S_\psi f)(x,y) 
	&= (S_\psi f)(x-a,y)  
	= \int_{G\times Y}
	f(u,v)\, \psi(x-a-u,y,v)\,
	\dif\nu(u)\,\dif\la(v).
\end{align*}
Using the change of variable
$t \eqdef u+a$
and the invariance of the Haar measure $\nu$,
we get
\begin{align*}
	(\rho(a)S_\psi f)(x,y)  
	&=
    \int_{G\times Y}
	f(t-a,v)\, \psi(x-t,y,v)\,
	\dif\nu(t)\,\dif\la(v)
    \\[0.5ex]
	&=
    \int_{G\times Y}
	(\rho(a)f)(t,v)\, \psi(x-t,y,v)\,
	\dif\nu(t)\,\dif\la(v)
    \\[0.5ex]
	&= (S_\psi \rho(a)f)(x,y).
\end{align*}
Thus, for every $a$ in $G$,
we have proved that $\rho(a)S_\psi=S_\psi\rho(a)$.
\end{proof}

\begin{prop}[injectivity of $\psi\mapsto S_\psi$]
\label{prop:injectivity_psi_to_Spsi}
Let $\psi\in\cA$ and $S_\psi=0$.
Then, $\psi=0$.
\end{prop}

\begin{proof}
It follows from Lemma~\ref{lem:tildeS_K_eq_psi}:
\[
\psi(x,y,v)
=(\tS_\psi K_{0,v})(x,y)
=(S_\psi K_{0,v})(x,y)
=0.
\qedhere
\]
\end{proof}

\begin{rem}
\label{rem:injectivity2}
Without the assumption~\eqref{eq:psi_cond_2}
included in the definition of $\cA$,
an analog of Proposition~\ref{prop:injectivity_psi_to_Spsi} may fail.
We use notation $\cL^2(G\times Y)$ and $\cZ(G\times Y)$
from Remark~\ref{example:M_without_second_condition},
with $X=G\times Y$.
We suppose that $H$,
after passing from functions to classes of equivalence,
does not coincide with $L^2(G\times Y)$:
\[
\bigl\{f+\cZ(G\times Y)\colon\ f\in H\bigr\} \ne L^2(G\times Y).
\]
Applying the theorem about the orthogonal decomposition,
we choose $h$ in $\cL^2(G\times Y)$
such that $h\notin\cZ(G\times Y)$
and $h$ is orthogonal to $H$.
Define $\psi$ by
\[
\psi(x,y,v) \eqdef \conj{h(-x,v)}
\qquad(x\in G,\ y,v\in Y).
\]
For every $(x,y)$ in $G\times Y$,
consider $g_{x,y}\eqdef\conj{\psi(x-\cdot,y,\cdot)}$,
i.e.,
\[
g_{x,y}(u,v)
=\conj{\psi(x-u,y,v)} 
=h(u-x,v).
\]
The properties of $h$ 
and the invariance of $H$ with respect to $\rho$ imply that
$g_{x,y}\in\cL^2(G\times Y)$ and $g_{x,y}\perp H$.
We define $S_\psi$ by~\eqref{eq:S_def}.
Then, for every $f$ in $H$,
\[
(S_\psi f)(x,y)
= \int_{G\times Y} f\,\conj{g_{x,y}}\,\dif\,(\nu\otimes\la)
= 0.
\]
Thus, $S_\psi$ is the zero operator.
\end{rem}

\begin{prop}
\label{prop:integral_form_for_operators_belonging_to_the_centralizer}
Let $T\in\cC(\rho)$.
Define $\psi\colon G\times Y\times Y\to\bC$ by
\[
\psi(u,v,y) \eqdef (T K_{0,y})(u,v)
\qquad(u\in G,\ v,y\in Y).
\]
Then, $\psi\in\cA$ and $T=S_\psi$.
\end{prop}

\begin{proof}
We apply Proposition~\ref{prop:initial_integral_form_of_bouned_invariant_operator} and notice that
\[
\psi(u,v,y)
=K_T(u,v,0,y).
\]
Hence, $\psi\in\cA_0$ by \eqref{eq:KT_1_2_belongs_to_H}.
By~\eqref{eq:initial_integral_form_of_bouned_invariant_operator} and~\eqref{eq:KT_invariant_property},
we get $\ttS f = T f$ for every $f\in H$.
Since $T\in\cB(H)$,
we conclude that $\psi\in\cA$
and $S_\psi=T$.
\end{proof}

\begin{rem}[$\cA$ is a vector space]
\label{rem:cA_is_vector_space}
We denote by $\bC^{G\times Y\times Y}$
the complex vector space
consisting of all complex-valued functions defined on
$G\times Y\times Y$, with pointwise operations.
It follows from the linear property of the integral
that $\cA$ is a vector subspace of $\bC^{G\times Y\times Y}$.
\end{rem}

Corollary~\ref{cor:Vb_eq_Sphib}
and Lemma~\ref{lem:S_psi_in_centralizer}
allow us to define the following functions.

\begin{defn}
\label{def:main_bijections}
We define
$\Phi\colon L^\infty(\Om)\to\cA$
and
$\Tht\colon \cA\to \cC(\rho)$
by
\[
\Phi(b) \eqdef \phi_b\qquad(b\in L^\infty(\Om)),
\]
\[
\Tht(\psi) \eqdef S_\psi\qquad(\psi\in\cA).
\]
\end{defn}

\begin{thm}
\label{thm:main_bijections}
$\Phi$ and $\Tht$ are bijective and linear,
and $\Tht\circ\Phi=\La$.
\end{thm}

\begin{proof}
From the linear property of the integral,
it follows that $\Phi$ and $\Tht$ are linear transformations.
Corollary~\ref{cor:Vb_eq_Sphib} yields
$\Tht\circ\Phi=\La$.
The invertibility of $\La$ implies that $\Tht$ is surjective,
and Proposition~\ref{prop:injectivity_psi_to_Spsi} implies that $\Tht$ is injective.
Moreover, by Proposition~\ref{prop:integral_form_for_operators_belonging_to_the_centralizer},
the inverse function $\Tht^{-1}\colon\cC(\rho)\to\cA$
acts by the following explicit formula:
\begin{equation}
\label{eq:Tht_inverse}
(\Tht^{-1}T)(u,v,y) = (T K_{0,y})(u,v)\qquad(u\in G,\ v,y\in Y).
\end{equation}
Now, from $\Tht\circ\Phi=\La$ we get
$\Phi=\Tht^{-1}\circ\La$.
Being a composition of invertible functions, $\Phi$ is invertible.
\end{proof}

Figure~\ref{fig:main_bijections} shows the main objects
of Theorem~\ref{thm:main_bijections}.

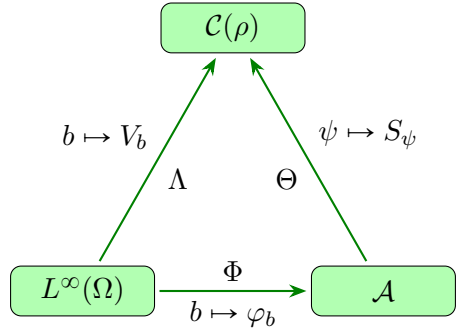
\begin{figure}[hbt]
\centering
\begin{tikzpicture}
	\node [spacenode] (Linfty)
	at (-2, 0) {$L^\infty(\Om)$};
	\node [spacenode] (cA)
	at (2, 0) {$\cA$};
	\node [spacenode] (centralizer)
	at (0, 3.5) {$\cC(\rho)$};
	% isometric isomorphisms
	\draw [myedge]
	(Linfty) edge (cA);
	\node at (0, 0) [above]
	{$\Phi$};
	\node at (0, 0) [below]
	{$b\mapsto\phi_b$};
	\draw [myedge]
	(Linfty) edge (centralizer);
	\node at (-1, 1.75) [above left]
	{$b\mapsto V_b$};
    \node at (-1, 1.75) [below right]
    {$\La$};
	\draw [myedge]
	(cA) edge (centralizer);
	\node at (1, 1.75) [above right]
	{$\psi \mapsto S_\psi$};
	\node at (1, 1.75) [below left]
	{$\Theta$};
\end{tikzpicture}
\caption{
\label{fig:main_bijections}
Spaces $L^\infty(\Om)$, $\cA$, and $\cC(\rho)$,
and bijections between them.
}
\end{figure}

\begin{cor}
\label{cor:descriptions_of_the_main_sets}
The sets $L^\infty(\Om)$, $\cA$, and $\cC(\rho)$
are related in the following manner:
\begin{align}
\label{eq:cA_in_terms_of_phi}
\cA &= \{\phi_b\colon\ b\in L^\infty(\Om)\},
\\
\label{eq:centralizer_in_terms_of_Spsi}
\cC(\rho) &= \{S_\psi\colon\ \psi\in\cA\}.
\end{align}
Moreover, $\cA$ can be written as the image of $\Tht^{-1}$,
where $\Tht^{-1}$ acts by~\eqref{eq:Tht_inverse}:
\begin{equation}
\label{eq:cA_in_terms_of_Tht_inverse}
\cA = \{\Tht^{-1}(S)\colon\ S\in\cC(\rho)\}.
\end{equation}
\end{cor}

\begin{cor}
\label{cor:Theta_inverse_adjoint_operator}
For every $T$ in $\cC(\rho)$,
every $u$ in $G$ and every $v,y$ in $Y$,
\begin{equation}
\label{eq:Theta_inverse_adjoint_operator}
(\Tht^{-1} T^\ast)(u,v,y)
= \conj{(T K_{0,v})(-u, y)}
= \conj{(\Tht^{-1}T)(-u,y,v)}.
\end{equation}
\end{cor}

\begin{proof}
We apply~\eqref{eq:Tht_inverse},
the reproducing property,
the definition of the adjoint operator,
and the assumption~\eqref{eq:K_is_invariant}:
\begin{align*}
   (\Tht^{-1} T^\ast)(u,v,y) 
   &=  (T^\ast K_{0,y})(u,v) 
   =  \langle T^\ast K_{0,y}, K_{u,v}\rangle  
   =  \langle K_{0,y}, TK_{u,v}\rangle
   \\[0.5ex]
   &=  \langle K_{0,y}, T\rho(u)K_{0,v}\rangle  
   =  \langle K_{0,y}, \rho(u) TK_{0,v}\rangle  
   =  \langle \rho(-u) K_{0,y}, TK_{0,v}\rangle
   \\[0.5ex]
   &=  \langle K_{-u,y}, TK_{0,v}\rangle  
   =  \overline{\langle TK_{0,v}, K_{-u,y}\rangle}  
   =  \overline{(TK_{0,v})(-u,y)}
   \\[0.5ex]
   &=  \overline{(\Tht^{-1}T)(-u,y,v)}.
   \qedhere
\end{align*}   
\end{proof}

In the next proposition
we provide an explicit construction for $\Phi^{-1}$.
It is similar to Proposition~\ref{prop:Lambda_inverse}.

\begin{prop}
\label{prop:Phi_inverse}
Let $\psi\in\cA$.
Then, the following equality holds in $L^\infty(\Om)$:
\begin{equation}
\label{eq:Phi_inverse}
(\Phi^{-1}\psi)(\xi)
=\frac{(R\psi(\cdot,\cdot,y))(\xi)}{\conj{q_\xi(y)}}
\qquad(\xi\in\Om),
\end{equation}
where $y$ in $Y_0$ such that $\xi\in\Om_y$.
\end{prop}

\begin{proof}
Let $Y_0$ and $(\Om_y)_{y\in Y_0}$
be as in Proposition~\ref{prop:partition_of_Omega}.
The right-hand side of~\eqref{eq:Phi_inverse},
which we denote by $b$,
is defined similarly to the construction
in the proof of Proposition~\ref{prop:Lambda_inverse}.
For every $y$ in $Y_0$,
we use the assumption that
$\psi(\cdot,\cdot,y)$ belongs to $H$.
We apply $R$ to $\psi(\cdot,\cdot,y)$,
restrict the obtained element of $L^2(\Om)$ to $\Om_y$,
and divide over $\conj{q_{\bullet}(y)}$ restricted to $\Om_y$.
Finally, we ``gather'' $b$ from these ``pieces'',
where $y$ runs through $Y_0$.

Let us show that $\phi_b=\psi$.
Since $\psi \in \cA$,
there exists $c\in L^\infty(\Om)$ such that $\psi=\Phi(c)=\phi_c$.
Since $\psi(\cdot, \cdot, y) \in H$,
we can apply $R$.
The following equality holds for almost all $\xi$ in $\Om$:
\[
(R\psi(\cdot,\cdot, y))(\xi) = c(\xi) \conj{q_\xi(y)}.
\] 
It is enough to show that $c$ coincides with $b$ on every $\Om_y,\, y\in Y_0$.
We fix $y_0\in Y_0$.
Then, $q_\xi(y_0)\neq 0$ for all $\xi$ in $\Om_{y_0}$.
So, for almost all $\xi$ in $\Om_{y_0}$, the above equation gives
\begin{equation*}
(\Phi^{-1}\psi)(\xi)
= c(\xi)
=\frac{(R\psi(\cdot,\cdot,y_0))(\xi)}{\conj{q_\xi(y_0)}}
= b(\xi).
\end{equation*}
As $y_0$ is an arbitrary element of $Y_0$
and $Y_0$ is finite or countable,
it follows that $b=c$ for almost every $\xi$ in $\Om$.
\end{proof}

\section{Special classes of translation-invariant operators}
\label{sec:integral_representation_special_operators}

In this section,
we provide an integral representation of the form $S_\psi$
for operators belonging to some important subclasses of $\cC(\rho)$.

\begin{prop}[Integral representation of translation-invariant operators with spectral symbols in the Wiener--Fourier algebra]
	Let $h\in L^1(G)$ and $b=Fh$.
	Then,
	\begin{equation}
		\label{eq:phi_for_Wiener_case}
		\phi_b(x,y,v)
		=
		\int_{G} h(u)\, K_{0,v}(x-u,y)\, \dif\nu(u).
	\end{equation}
\end{prop}

\begin{proof}
Given $v,y$ in $Y$, we define $g_{v,y}$ by
\[
g_{v,y}(\cdot) \eqdef h \ast K_{0,v}(\cdot, y).
\]
In other words, $g_{v,y}(x)$ is the right-hand side of~\eqref{eq:phi_for_Wiener_case}.
By assumption~\eqref{eq:K_tecnhical_assumption},
we have $K_{0,v}(\cdot,y)\in L^1(G)$.
Since $L^1(G)$ is closed under the convolution operation,
$g_{v,y}\in L^1(G)$.
By the convolution theorem,
\[
(Fg_{v,y})(\xi)
= (Fh)(\xi) \, (FK_{0,v}(\cdot,y))(\xi)
= b(\xi) L_{\xi,v}(y).
\]
Due to Remark~\ref{rem:integral_in_def_phi_exists_in_Lebesgue_sense},
we know that $Fg_{v,y}\in L^1(\Om)$.
Therefore, we can apply the Fourier inversion theorem:
\[
   g_{v,y}(x)
   = \int_{\widehat{G}} \xi(x)\, (Fg_{v,y})(\xi)
   = \int_{\hG} \xi(x)\, b(\xi)\, L_{\xi,v}(y)\, \dif\hnu(\xi)
   = \phi_b(x,y,v).
\qedhere
\]
\end{proof} 

\begin{prop}[Integral representation of horizontal translations]
Let $a\in G$.
Then, $\rho(a)=S_\psi$, where
\begin{equation}
\psi(x,y,v)
=K_{a,v}(x,y).
\end{equation}
\end{prop}

\begin{proof}
We know that $\rho(a)\in\cC(\rho)$.
By 
Proposition~\ref{prop:integral_form_for_operators_belonging_to_the_centralizer},
$\rho(a)=S_\psi$, where
\[
\psi(x,y,v)
=(\rho(a) K_{0,v})(x,y)
=K_{a,v}(x,y).
\]
Let us verify the equality $\rho(a)=S_\psi$ directly.
For $f$ in $H$ and $(x,y)$ in $Y$,
\[
	(\rho(a)f)(x,y)
    = \langle \rho(a)f, K_{x,y}\rangle 
	=\langle f, \rho(-a)K_{x,y} \rangle 
	= \langle f, K_{x-a,y}\rangle.
\]
Since
\[
\conj{K_{x-a,y}(u,v)}
=
K_{u,v}(x-a,y)
=
K_{a,v}(x-u,y)
=\psi(x-u,y,v),
\]
we get
\[
(\rho(a)f)(x,y)
= \int_{G\times Y} f(u,v)\, \psi(x-u,y,v)\, \dif\nu(u)\, \dif\la(v)
= (S_\psi f)(x,y).
\qedhere
\]
\end{proof}

\begin{example}[Integral representation of Toeplitz operators with vertical symbols]
Let $g\in L^\infty(Y)$.
Define $\widetilde{g}\colon G\times Y\to\bC$,
\[
\widetilde{g}(u,v)\eqdef g(v),
\]
and consider the Toeplitz operator
$T_{\widetilde{g}}$.
From \cite{HerreraMaximenkoRamos2022},
we know that
$R T_{\widetilde{g}} R^\ast=M_{\ga_g}$,
where
\[
(\ga_g)(\xi)
\eqdef
\int_Y g(t)\, |q_\xi(t)|^2\,\dif\la(t).
\]
By Theorem~\ref{thm:V_integral_representation}, we have 
\begin{align*}
T_{\widetilde{g}}
= R^\ast M_{\ga_g}R = S_{\phi_{\ga_g}},
\end{align*}
where $\phi_{\ga_g}$ is given by \eqref{eq:integral_kernel_from_spectral_function} with $b$ replaced by $\ga_g$.

On the other hand,
by the definition of Toeplitz operators and the reproducing property,
	\begin{align*}
		(T_{\widetilde{g}}f)(x,y)
		&=\langle \widetilde{g}f, K_{x,y}\rangle_{L^2(G\times Y)}
		=
		\int_{G\times Y}
		f(u,v)\, g(v)\,
		\overline{K_{x,y}(u,v)}\,
		\dif\nu(u)\,
		\dif\la(v)
		\\[0.5ex]
		&=
		\int_{G\times Y}
		f(u,v)\, g(v)\,
		\overline{K_{0,y}(u-x,v)}\,
		\dif\nu(u)\,
		\dif\la(v).
	\end{align*}
Let $\psi_1(x,y,v)\eqdef g(v)\,\conj{K_{0,y}(-x,v)}$.
In general, $\psi_1$ does not coincide with $\phi_{\ga_g}$,
because it, in general,
does not satisfy~\eqref{eq:psi_cond_2}.
Thus, we have two different integral representations for the same operator.
\end{example}

\section{W*-algebra of functions}
\label{sec:algebra_of_functions}

In this section,
we construct a W*-algebra of functions on $G\times Y\times Y$.
In particular, we define a new multiplication operation, adjoint operation, and a norm on the set $\cA\subseteq G\times Y\times Y$ with respect to which it forms a commutative $W^\ast$-algebra.
With this structure on $\cA$,
functions $\Phi$ and $\Tht$ are isometric isomorphisms.

\begin{defn}
\label{def:cA_adjoint_operation}
Let $\psi\in\cA$.
Define $\psi^\dagger\colon G\times Y\times Y\to\bC$ by
\begin{equation}
\label{eq:cA_adjoint_operation}
\psi^\dagger(x,y,v)
\eqdef
\conj{\psi(-x,v,y)},
\end{equation}
for $x$ in $G$ and $y,v$ in $Y$.
\end{defn}

\begin{defn}
\label{def:cA_multiplication_operation}
Let $\psi,\eta\in\cA$.
Define $\psi\odot\eta\colon G\times Y\times Y\to\bC$ by
\begin{equation}
\label{eq:cA_mul_operation}
(\psi \odot \eta)(x,y,v)
\eqdef
\int_{G\times Y}
\psi(x-s,y,t)\,\eta(s,t,v)\,
\dif\nu(s)\, \dif\la(t),
\end{equation}
for $x$ in $G$ and $y,v$ in $Y$.
\end{defn}

\begin{defn}
\label{def:cA_norm}
For each $\psi$ in $\cA$,
define $\|\cdot\|_\cA$ on $\cA$ by
\begin{equation}
\label{eq:cA_norm}
\|\psi\|_{\cA} \eqdef \|S_\psi\|_{\cB(H)}.
\end{equation}
\end{defn}

Since $\Tht$ is an isomorphism of vector spaces $\cA$ and $\cC(\rho)$,
$\|\cdot\|_\cA$ is a norm on $\cA$.

\begin{lem}
\label{lem:cA_adjoint_operation_and_Phi}
For every $b$ in $L^\infty(\Om)$,
\begin{equation}
\label{eq:Phi_conj}
\Phi(\conj{b})=\Phi(b)^\dagger.
\end{equation}
\end{lem}

\begin{proof}
Let $b\in L^\infty(\Om)$
and $\psi\eqdef\Phi(b)$.
Notice that $\conj{b}$ belongs to $L^\infty(\Om)$.
Thus, $\Phi(\conj{b})$ makes sense.
Moreover, by Corollary~\ref{cor:Vb_eq_Sphib},
it belongs to $\cA$.
Let us verify~\eqref{eq:Phi_conj}.
For all $x$ in $G$ and $y,v$ in $Y$,
\begin{align*}
(\Phi(\conj{b}))(x,y,v)
&=
\phi_{\conj{b}}(x,y,v) 
=
\int_{\Om} \conj{b(\xi)}\, \xi(x)\, L_{\xi,v}(y)\, \dif\hnu(\xi)
\\[1ex]
&=
\conj{\int_{\Om} b(\xi)\, \xi(-x)\, L_{\xi,y}(v)\, \dif\hnu(\xi)}
= \conj{\psi(-x,v,y)}
= \psi^\dagger(x,y,v).
\qedhere
\end{align*}
\end{proof}

\begin{lem}
\label{lem:cA_adjoint_operation_and_Theta}
For every $\psi$ in $\cA$,
we have $\psi^\dagger\in\cA$ and
\begin{equation}
\label{eq:Tht_adjoint}
\Tht(\psi^\dagger)=\Tht(\psi)^\ast.
\end{equation}
\end{lem}

\begin{proof}[First proof]
Follows from Corollary~\ref{cor:Theta_inverse_adjoint_operator}.
\end{proof}

\begin{proof}[Second proof]
Given $\psi$ in $\cA$,
we apply Corollary~\ref{cor:descriptions_of_the_main_sets}
and find $b$ in $L^\infty(\Om)$ such that
$\Phi(b)=\psi$.
Then, by Lemma~\ref{lem:cA_adjoint_operation_and_Phi},
$\psi^\dagger
=\Phi(b)^\dagger
=\Phi(\conj{b})
\in\cA$.
Furthermore,
we apply the fact that $\La=\Tht\circ\Phi$
(Theorem~\ref{thm:main_bijections})
and $\La$ is an isometric isomorphism of C*-algebras
(Theorem \ref{thm:centralizer_through_Vb}):
\begin{align*}
\Tht(\psi^\dagger)
&=\Tht(\Phi(\conj{b}))
=(\Tht\circ\Phi)(\conj{b})
=\La(\conj{b})
=\La(b)^\ast
=(\Tht(\Phi(b))^\ast
=\Tht(\psi)^\ast,
\end{align*}
and we get~\eqref{eq:Tht_adjoint}.
\end{proof}

\begin{lem}
\label{lem:cA_mul}
Let $\psi,\eta\in\cA$.
Then, $\psi\odot\eta\in\cA$,
and $\Tht(\psi\odot\eta)=\Tht(\psi)\Tht(\eta)$.
\end{lem}

\begin{proof}
For every $x\in G$ and every $y,v\in Y$,
comparing~\eqref{eq:S_def} with~\eqref{eq:cA_mul_operation}, we get
\begin{equation}
\label{eq:cA_mul_operation_via_S}
(\psi\odot \eta)(x,y,v)
= (S_\psi \eta(\cdot,\cdot,v))(x, y).
\end{equation}
Furthermore,
by Corollary~\ref{cor:Vb_eq_Sphib},
we have $\eta(x,y,v) = S_{\eta} K_{0,v}(x,y)$.
We substitute this formula into~\eqref{eq:cA_mul_operation_via_S}
and express $\psi\odot\eta$ in terms of the operators
$S_\psi$ and $S_\eta$:
\begin{equation}
\label{eq:psi_mul_eta_via_S_psi_S_eta}
(\psi \odot \eta)(x, y, v)
= (S_\psi S_\eta K_{0, v})(x, y).
\end{equation}
On the other hand,
by Lemma~\ref{lem:S_psi_in_centralizer},
we have $S_\psi,S_\eta\in\cC(\rho)$.
Since $\cC(\rho)$ is an algebra,
$S_\psi S_\eta\in\cC(\rho)$.
By Theorem~\ref{thm:main_bijections}
and formula~\eqref{eq:Tht_inverse},
\begin{equation}
\label{eq:omega_explicit_via_S_psi_S_eta}
\Tht^{-1} (S_\psi S_\eta)(x,y,v)
= (S_\psi S_\eta K_{0,v})(x,y).
\end{equation}
It follows from~\eqref{eq:psi_mul_eta_via_S_psi_S_eta}
and~\eqref{eq:omega_explicit_via_S_psi_S_eta}
that
\begin{equation}
\label{eq:psi_mul_eta_via_Tht_inverse_S_psi_S_eta}
\psi\odot\eta = \Tht^{-1}(S_\psi S_\eta).
\end{equation}
Since $\Tht$ is a bijection between $\cA$ and $\cC(\rho)$,
we conclude that $\psi\odot\eta\in\cA$.
Moreover, applying $\Tht$ to both sides of~\eqref{eq:psi_mul_eta_via_Tht_inverse_S_psi_S_eta},
we obtain that $\Tht(\psi\odot\eta)=\Tht(\psi)\Tht(\eta)$.
\end{proof}

\begin{lem}
\label{lem:Phi_and_product}
Let $b,c\in L^\infty(\Om)$. Then,
\begin{equation}
\label{eq:Phi_product}
\Phi(bc)=\Phi(b)\Phi(c).
\end{equation}
\end{lem}

\begin{proof}
It follows easily from Lemma~\ref{lem:cA_mul}
and Theorems~\ref{thm:centralizer_through_Vb} and \ref{thm:main_bijections}.
Indeed,
\[
\Tht(\Phi(bc))
=\La(bc)
=\La(b)\La(c)
=\Tht(\Phi(b))\Tht(\Phi(c))
=\Tht(\Phi(b)\Phi(c)).
\]
Applying $\Tht^{-1}$ to both sides, we get $\Phi(bc)=\Phi(b)\Phi(c)$.
\end{proof}

\begin{prop}[transport of C*-algebra structure]
\label{prop:transport_Cstar_structure}
Let $\cE$ be C*-algebra,
$\cF$ be a set,
and $\Psi\colon\cE\to\cF$ be a bijection.
Define $\oplus\colon\cF^2\to\cF$,
$\odot\colon\bC\times\cF\to\cF$,
$\otimes\colon\cF^2\to\cF$,
$\dagger\colon\cF\to\cF$,
$\|\cdot\|_{\cF}\colon\cF\to[0,+\infty)$ by
\begin{align*}
a \oplus b &\eqdef \Psi(\Psi^{-1}(a)+\Psi^{-1}(b)),
&
\la \odot a &\eqdef \Psi(\la \Psi^{-1}(a)),
\\
a \otimes b &\eqdef \Psi(\Psi^{-1}(a) \times \Psi^{-1}(b)),
&
a^{\dagger} &\eqdef \Psi(\Psi^{-1}(a)^\ast),
\\
\|a\|_{\cF} & \eqdef \|\Psi^{-1}(a)\|_{\cE},
\end{align*}
where $+$, $\cdot$, $\times$, $\ast$, and $\|\cdot\|_{\cE}$
are the corresponding operations and norm on $\cE$.
Then, $\cF$, considered with these operations and function $\|\cdot\|_{\cF}$,
is a C*-algebra, and $\Psi$ is an isometric isomorphism of C*-algebras.
If $\cE$ is commutative, then $\cF$ also is commutative.
If $\cE$ is a W*-algebra, then $\cF$ also is a W*-algebra.
\end{prop}

\begin{proof}
The transport of structures is a well-known idea in mathematics
(see, e.g., a general treatment in
Bourbaki~\cite[Chapter~IV, Section~5]{Bourbaki1968}),
and the proof of this proposition is very simple.
Let us only check the properties
$(a\otimes b)^\dagger=(b^\dagger)\otimes(a^\dagger)$
and $\|a\otimes b\|_{\cF}\le\|a\|_{\cF}\|b\|_{\cF}$.
If $x=\Psi^{-1}(a)$ and $y=\Psi^{-1}(b)$, then
\[
(a\otimes b)^\dagger
= (\Psi(x\times y))^\dagger
= \Psi((x\times y)^\ast)
= \Psi(y^\ast \times x^\ast)
= \Psi(y^\ast) \otimes \Psi(x^\ast)
= b^\dagger \otimes a^\dagger
\]
and
\[
\|a\otimes b\|_{\cF}
= \|\Psi(x\times y)\|_{\cF}
= \|x\times y\|_{\cE}
\le \|x\|_{\cE}\,\|y\|_{\cE}
= \|a\|_{\cF}\,\|b\|_{\cF}.
\qedhere
\]
\end{proof}

We consider $\cA$ with the operations $\odot$ and $\dagger$
(see Definitions~\ref{def:cA_adjoint_operation}
and~\ref{def:cA_multiplication_operation}),
and the norm $\|\cdot\|_\cA$ from Definition~\ref{def:cA_norm}.

Now we are ready to state our main theorem of this section.  

\begin{thm}
$\cA$ is a commutative W*-algebra.
$\Phi\colon L^\infty(\Om)\to\cA$
and $\Tht\colon\cA\to\cC(\rho)$
are isometric isomorphisms of W*-algebras.
\end{thm}

\begin{proof}
It follows from Lemmas~\ref{lem:cA_adjoint_operation_and_Phi}--\ref{lem:Phi_and_product}
and Proposition~\ref{prop:transport_Cstar_structure}.
\end{proof}

\begin{rem}[General form of projections in $\cA$]
Given a measurable subset $E$ of $\Om$,
we denote its characteristic function by $\charfun_E$.
For $b=\charfun_E$,
\eqref{eq:integral_kernel_from_spectral_function}
takes the following form:
\begin{equation}
\label{eq:phi_char_fun}
\phi_{\charfun_E}(x,y,v)
=
\int_{E}
\xi(x)\, L_{\xi,v}(y)\,\dif\hnu(\xi).
\end{equation}
It is well known and easy to see that every projection in the W*-algebra $L^\infty(\Om)$ is of the form $\charfun_E$ for some $E$
(indeed, if $f^2=f$ almost everywhere, then $f(\xi)\in\{0,1\}$
for almost all $\xi$ in $\Om$).
Since $\Phi$ is an isomorphism of W*-algebras,
\eqref{eq:phi_char_fun}
is the general form of projections in $\cA$.
\end{rem}

\begin{rem}
    Let $x,y \in G$ and $v,y\in Y$.  For $\psi \in \cA$, we define $f_\psi: G \times Y \times G\times Y$ by
    \begin{align*}
        f_\psi(x,y,u,v) \eqdef \psi(x-u,y,v).
    \end{align*}
    Then, it can be shown that $\{f_\psi \colon \psi \in \cA\}$ forms a commutative C*-algebra of functions on $ G \times Y \times G\times Y$.  Applying this technique to examples in Section \ref{sec:examples}, we can construct several C*-algebras of analytic  functions and C*-algebras of harmonic functions on various subsets of $\mathbb{C}$.  In particular, we obtain the C*-algebras constructed in \cite{Bais_PAMS_2024, Mohan-Venku_2024, Ma-Zhu_PAMS_2024}.
\end{rem}

\begin{rem}[Extension of the new multiplication formula]
    If we do not care about the C*-algebra structure, then it is natural to look for sets $\mathcal{S}\subseteq \big\{f: G\times G\times Y\rightarrow \bC\big\}$ such that $\cA \subseteq \mathcal{S}$ and $\odot: \mathcal{S}\times \mathcal{S}\rightarrow \mathcal{S}$ is well-defined.  

    We observe that if $\psi \in \cA$ and $\eta \in \cA_0$, then $\psi \odot \eta = \eta \odot \psi $ and it belongs to $\cA_0$.  That is, $\odot : \cA \times \cA_0 \rightarrow \cA_0$ and $\odot : \cA_0 \times \cA \rightarrow \cA_0$ are well-defined.  Now, it is interesting to check if we can extend $\odot$ from $\cA_0 \times \cA_0 \rightarrow \cA_0.$  
\end{rem}

Recently, in \cite{Karapetyants_2020_convolution_disk, Karapetyants_2022_convolution_disk, Karapetyants_2024_convolution_UHP},
multiplication operations in some special cases are considered with different motivations and their mapping properties are discussed.

\section{Examples}
\label{sec:examples}

In this section,
we recall some of the well-known W*-algebras of operators
and provide integral representation for operators in them.
Throughout this section, we use the following conventions. 
We suppose that $\hG$ is a locally compact abelian group
\emph{topologically isomorphic} to the dual group of $G$.
Let $E(\cdot, \cdot)\colon G\times\hG\to\bT$
be the duality pairing between $G$ and $\hG$
(see, e.g., Folland~\cite[Section~4.1]{Folland2016}).
The Haar measures $\nu$ and $\hnu$ are chosen in such a way
that the Fourier transform $F$
is a unitary operator from $L^2(G)$ onto $L^2(\hG)$.
In particular, for $G=\bR$,
we employ $\hG=\bR$,
the usual Lebesgue measure $\nu=\hnu$, 
and $E(x,\xi) = \enumber^{2\pi \iu x\xi}$.
Furthermore,
we use the following notations:
$\bR_+ \eqdef (0,\infty)$,
$\bNz \eqdef \{0,1,2,\dots\}$,
$\bT \eqdef \{z\in\bC\colon\ \vert z\vert=1\}$,
$\charfun_Z$
is the characteristic function of the set $Z$.

\begin{example}[Vertical operators in the Bergman space over the upper half-plane]
\label{example:vertical_Bergman_UHP}
This example is inspired by
Vasilevski~\cite[Sections~3.1 and 5.2]{Vasilevski2008book},
see also
\cite[Example~9.1]{HerreraMaximenkoRamos2022}
and \cite{BaisNaiduPinlodi2023Bergman}.
Let $\Pi = \{z\in \bC\colon\ \Im(z)>0\}$ be the upper half-plane and $L_{\mathrm{hol}}^2(\Pi)$ be the Bergman space on $\Pi$
which consists of all analytic functions in $L^2(\Pi)$.
It is a RKHS with the reproducing kernel given by
\[
K(z,w)
\eqdef K_w(z)
= -\frac{1}{\pi(z-\overline{w})^2}\qquad (z,w\in \Pi).
\]
We identify $\Pi$ with $\bR\times \bR_+$.
In this example,
$G=\hG=\bR$, $Y=\bR_+$,
$\nu = \hnu$ is the Lebesgue measure on $\bR$,
$\la$ is the restriction of the Lebesgue measure of $\bR$ to $\bR_+$,
and $H=L_{\mathrm{hol}}^2(\Pi)$.
Furthermore, $(\rho(a)f)(z)=f(z-a)$,
and the elements of $\cC(\rho)$ are called \emph{vertical operators}.
By identifying $z$ with $(x,y)$ and $w$ with $(u,v)$,
we rewrite the reproducing kernel as
\[
K_{u,v}(x,y)
=-\frac{1}{\pi\bigl((x-u)+\iu\,(y+v)\bigr)^2}.
\]
The horizontal Fourier transform of $K_{0,v}(\cdot,y)$ is easy to compute using residues:
\begin{equation}
\label{eq:L_Bergman_vertical}
L_{\xi,v}(y)
=-\frac{1}{\pi}
\int_{\bR}
\frac{\enumber^{-2\pi\iu\xi u}\,\dif{}u}{\bigl(u+\iu (y+v)\bigr)^2}
=
\begin{cases}
4\pi\xi\enumber^{-2\pi (y+v)\xi}, & \xi > 0;
\\
0, & \xi \le 0.
\end{cases}
\end{equation}
Thus, in this example,
$\Om=\bR_+$ and
\[
q_\xi(y)=2\sqrt{\pi \xi}\,\enumber^{-2\pi y \xi}.
\]
Given $b$ in $L^\infty(\bR_+)$,
\begin{equation}
\label{eq:phi_vertical_Bergman}
\begin{aligned}
\phi_b(x,y,v)
&=
4\pi \int_{0}^{+\infty}
\xi \enumber^{-2\pi (y+v)\xi}\,
\enumber^{2\pi\iu x \xi}\, b(\xi)\, \dif\xi
\\[0.5ex]
&= 4\pi \int_{0}^{+\infty} \xi \enumber^{2\pi\iu\, (x+\iu y)\xi}\, \enumber^{-2\pi v\xi}\, b(\xi)\, \dif\xi.
\end{aligned}
\end{equation}
The operator $V_b=S_{\phi_b}$ takes the form
\[
(V_b f)(x,y)
=
4\pi\int_{\bR\times\bR_+} f(u,v)
\Bigl(\int_{0}^{+\infty} \xi \,b(\xi)\,
\enumber^{2\pi\iu\, ((x+\iu y)-(u-iv))\xi}\, \dif\xi\Bigr)\,
\dif u \dif v.
\]
This representation is similar to~\cite{BaisNaiduPinlodi2023Bergman}.
\end{example}

\begin{rem}[$\cA_0$ does not coincide with $\cA$]
\label{rem:A0_is_not_A}
In the context of Example~\ref{example:vertical_Bergman_UHP},
we will construct a function $\psi$ such that $\psi\in\cA_0\setminus\cA$.
Let $c(\xi) \eqdef \xi$ for every $\xi$ in $\bR_+$.
We define $\psi$ similarly to~\eqref{eq:phi_vertical_Bergman},
but, instead of $b$ in $L^\infty(\bR_+)$,
we take the unbounded function $c$:
\[
\psi(x,y,v)
\eqdef
4\pi \int_{\bR_+}
\xi c(\xi)
\enumber^{2\pi \iu\,(x+\iu y)\xi}
\enumber^{-2\pi v\xi}\,\dif{}\xi.
\]
The integral can be computed explicitly.
Moreover, for a fixed $v>0$,
the function $(x,y)\mapsto \psi(x,y,v)$
is analytic with respect to $z=x+\iu y$
and belongs to $L_{\mathrm{hol}}^2(\Pi)$:
\[
\psi(x,y,v)
=
4\pi
\int_{\bR_+}
\xi^2 \enumber^{-2\pi(y+v-\iu x)\xi}\,\dif{}\xi
=\frac{1}{\pi^2 (y+v-\iu x)^3}
=-\frac{\iu}{\pi^2 (z + \iu v)^3}.
\]
Furthermore, for a fixed $y>0$,
$(u,v)\mapsto\conj{\psi(-u,y,v)}$
is analytic with respect to $w=u+\iu v$
and belongs to $L_{\mathrm{hol}}^2(\Pi)$:
\[
\conj{\psi(-u,y,v)}
=\psi(u,v,y)
=-\frac{\iu}{\pi^2 (w+\iu y)^3}.
\]
Therefore, $\psi\in\cA_0$.
Let us prove that $\psi\notin\cA$.
Suppose that $b\in L^\infty(\bR_+)$ and $\psi=\phi_b$.
Then,
\[
4\pi \int_{\bR}
f_{y,v}(\xi)\,\enumber^{2\pi\iu x \xi}\,\dif{}\xi=0,
\]
where
\[
f_{y,v}(\xi)
\eqdef
\charfun_{\bR_+}(\xi)  (b(\xi)-\xi) \xi \enumber^{-2\pi (y+v)\xi}.
\]
It is easy to see that $f_{y,v}\in L^1(\bR)\cap L^2(\bR)$.
Thus, by the injective property of the Fourier transform
(or by the isometric property of the Fourier--Plancherel transform),
$f_{y,v}=0$ almost everywhere.
Therefore, $b(\xi)=\xi$ for almost every $\xi$ in $\bR_+$,
which contradicts to the assumption that $b\in L^\infty(\bR_+)$.
\end{rem}

\begin{example}[Vertical operators in the harmonic Bergman space]
\label{example:vertical_harmonic_Bergman_UHP}
This example is based on~\cite[Example~9.2]{HerreraMaximenkoRamos2022},
see also Loaiza and Lozano~\cite{LoaizaLozano2013}.
Let $G, Y, \nu, \lambda$ be as in Example~\ref{example:vertical_Bergman_UHP}.  The harmonic Bergman space $L_{\mathrm{harm}}^2(\Pi)$ consists of all harmonic functions in $L^2(\Pi)$.
It is well known that
$L_{\mathrm{harm}}^2(\Pi)
= L_{\mathrm{hol}}^2(\Pi) \bigoplus \conj{L_{\mathrm{hol}}^2(\Pi)}$.
It is an RKHS, and its reproducing kernel is
\[
K(z,w) = K_w(z)
= -\frac{1}{\pi(z-\overline{w})^2} -\frac{1}{\pi(\overline{z}-w)^2}.
\]
Similar to Example~\ref{example:vertical_Bergman_UHP},
after identifying $z$ with $(x,y)$ and $w$ with $(u,v)$,
\[
K_{u,v}(x,y)
=-\frac{1}{\pi\bigl((x-u)+\iu\,(y+v)\bigr)^2}
-\frac{1}{\pi\bigl((x-u)-\iu\,(y+v)\bigr)^2}.
\]
Therefore, similarly to~\eqref{eq:L_Bergman_vertical},
\begin{equation}
\label{eq:L_Bergman_harmonic_vertical}
L_{\xi,v}(y)
=-\frac{1}{\pi}
\int_{\bR}
\frac{\enumber^{-2\pi\iu\xi u}\,\dif{}u}{\bigl(u+\iu\,(y+v)\bigr)^2 + \bigl(u-\iu\,(y+v)\bigr)^2}
= 4\pi \vert \xi\vert\, \enumber^{-2\pi \vert \xi\vert (y+v)}.
\end{equation}
Here we have
$\Om=\bR$ and
\[
q_\xi(y)=2\sqrt{\pi \vert\xi\vert} \enumber^{-2\pi y \vert\xi\vert}.
\]
For $b$ in $L^\infty(\bR)$, we have 
\begin{equation}
\label{eq:phi_vertical_harmonic}
\phi_b(x,y,v)
= 4\pi\, \int_{\mathbb{R}} \vert\xi\vert\,
\enumber^{-2\pi y\vert\xi\vert-2\pi v\vert\xi\vert}\, \enumber^{2\pi\iu x \xi}\, b(\xi)\, \dif\xi.
\end{equation}
%\begin{align*}
				%		&= 4\pi\, \int_{0}^{\infty} \xi\, \enumber^{-2\pi y\xi-2\pi v\xi}\, \enumber^{2\pi\iu x \xi}\, b(\xi)\, \dif\xi 
				%		+
				%		4\pi\, \int_{-\infty}^{0} (-\xi)\, \enumber^{-2\pi y(-\xi)-2\pi v(-\xi)}\, \enumber^{2\pi\iu x \xi}\, b(\xi)\, \dif\xi  \\
				%		&= 4\pi\, \int_{0}^{\infty} \xi\, \enumber^{-2\pi y\xi-2\pi v\xi}\, \enumber^{2\pi\iu x \xi}\, b(\xi)\, \dif\xi 
				%		+
				%		4\pi\, \int_{0}^{\infty} \xi\, \enumber^{-2\pi y\xi-2\pi v\xi}\, \enumber^{-2\pi\iu x \xi}\, b(-\xi)\, \dif\xi  \\
				%		&= 4\pi\, \int_{0}^{\infty} \xi\, b(\xi)\, \enumber^{2\pi \iu\xi(x+iy)}\, \enumber^{-2\pi v\xi} + 4\pi\, \int_{0}^{\infty} \xi\, b(-\xi)\, \enumber^{-2\pi \iu \xi(x-iy)}\, \enumber^{-2\pi v\xi}\, \dif\xi
%\end{align*}
The operator $V_b=S_{\phi_b}$ takes the form 
\begin{equation}
\label{eq:V_b_harmonic_vertical}
(V_b f)(x,y) 
=  4\pi\, \int_{\mathbb{R}\times \mathbb{R}_+} f(u,v)\,
\Bigl(
\int_{\bR} \vert\xi\vert\,
\enumber^{2\pi\iu\, ((x+\iu y)-(u-iv))\,|\xi|}\,
b(\xi)\, \dif\xi
\Bigr)\,
\dif u \dif v.
\end{equation}
\end{example}

\begin{example}[Vertical operators in wavelet spaces]
This example is based on
Hutn\'{i}k and Hutn\'{i}kov\'{a}~\cite{HutnikHutnikova2011}
and~\cite[Example~9.5]{HerreraMaximenkoRamos2022}.
Let $\psi$ be a function belonging to $L^2(\bR)$
and satisfying the following admissibility condition:
\[
\int_{\bR_+} \vert (F\psi)(t\xi)\vert^2\, \frac{\dif t}{t} =1
\quad(\xi \in \bR\setminus\{0\}),\quad (F\psi)(0)=0.
\]
Let $G=\hG=\bR$ with Lebesgue measure
and $Y=\bR_+$ with measure $\dif\la(y) = \frac{\dif y}{y^2}$.
The set $G\times Y$ can be identified with the positive affine group.
For each $(x,y)$ in $\bR\times \bR_+$, we define
\[
\psi_{x,y}(t) \eqdef \frac{1}{\sqrt{y}}\, \psi\left(\frac{t-x}{y}\right).
\]
Let $W_\psi\colon L^2(\bR) \to L^2(\bR\times \bR_+, \nu\times\la)$
be the \emph{continuous wavelet transform} associated to the wavelet $\psi$:
\[
W_\psi f(x,y) \eqdef \langle f, \psi_{x,y}\rangle_{L^2(\bR)}.
\]
We consider the \emph{wavelet space} $H$
defined as the image of $L^2(\bR)$ under the map $W_\psi$.
It is a reproducing kernel Hilbert space with reproducing kernel given by
\[
K_{u,v}(x,y)
=
\langle \psi_{x,y},\psi_{u,v} \rangle_{L^2(\bR)}
= \langle \psi_{x-u,y},\psi_{0,v} \rangle_{L^2(\bR)}.
\]
As it was shown in \cite[Example~9.5]{HerreraMaximenkoRamos2022},
\begin{equation}
\label{eq:L_wavelet_vertical}
L_{\xi,v}(y) = \sqrt{yv}\, (F\psi)(y\xi)\, \overline{(F\psi)(v\xi)}\qquad (\xi\in \bR,\, y,v\in \bR_+).
\end{equation}
Therefore, $\Om = \bR$ and $q_\xi(y) = \sqrt{y}\, (F\psi)(y\xi)$.
For $b$ in $L^\infty(\bR)$, we have
\begin{equation}
\label{eq:integral_representation_in_wavelet_space}
\phi_b(x,y,v)
=  \sqrt{yv}\,
\int_{\bR} b(\xi)\,
(F\psi)(y\xi)\,\conj{(F\psi)(v\xi)}\,
\enumber^{2\pi\iu x\xi}\, \dif\xi.
\end{equation}
In this example,
the operator
$V_b=S_{\varphi_b}$ takes the form
\[
(V_b f)(x,y)
=
\int_{\bR\times \bR_+} f(u,v)\, \Bigl(\int_{\bR} \sqrt{yv}\, b(\xi)\, (F\psi)(y\xi)\, \conj{(F\psi)(v\xi)}\,
\enumber^{2\pi\iu (x-u)\xi}\, \dif\xi\Bigr)\, \frac{\dif u\, \dif v}{v^2}. 
\]
\end{example}

\begin{example}[Translation-invariant operators in the Steinwart--Hush--Scovel space]
\label{Example_Vertical_RBFK}
This example is based on~Steinwart, Hush, and Scovel~\cite{SteinwartHushScovel2006}
and~\cite[Example~9.11]{HerreraMaximenkoRamos2022}.
Let $\al>0$.
Consider the space $H$ of all holomorphic functions $f$ on $\bC^n$ such that
\[
\Vert f \Vert_H^2
\eqdef
\frac{2^n\al^{2n}}{\pi^n}
\int_{\bC^n} \vert f(z) \vert^2\,
\enumber^{-4\al^2 \sum_{j=1}^{n}\Im(z_j)^2}\,\dif\mu_{\bC^n}(z) < +\infty,
\]
where $\mu_{\bC^n}$ is the Lebesgue measure on $\bC\cong\bR^{2n}$.
As it is shown in~\cite{SteinwartHushScovel2006},
$H$ is closely related to the Gaussian kernel
which is widely used in machine learning.
It is an RKHS with reproducing kernel given by
\[
K(z,w)
= K_{w}(z)
\eqdef \prod_{j=1}^{n}
\enumber^{-\al^2(z_j-\conj{w}_j)^2}.
\]
We now identify $\bC^n$ with $\bR^n\times \bR^n$.
Let $G=Y=\bR^n$,
$E(x,\xi) = \enumber^{2\pi \iu\,\langle x,\xi\rangle}$,
$\nu = \hnu$ be the Lebesgue measure on $\bR^n$,
and the measure on $Y=\bR^n$ be given by 
\[
\dif\la(v)
\eqdef
\frac{2^n\al^{2n}}{\pi^n}\,  \enumber^{-4\al^2 \|v\|^2}\, \dif\nu(v)
=
\prod_{j=1}^{n}
\biggl(
\frac{2\al^{2}}{\pi}\, \enumber^{-4\al^2 v_j^2}\,
\dif v_j
\biggr).
\]
After identifying $\bC^n$ with $\bR^n\times\bR^n$,
we get
\[
K_{u,v}(x,y)
= \prod_{j=1}^{n} \enumber^{-\al^2(x_j-u_j)^2+\al^2(y_j+v_j)^2 -2\iu\al^2(x_j-u_j)(y_j+v_j)} .
\]
The function $L_{\xi,v}(y)$
is easy to compute by using the Gaussian integral:
\begin{align*}
L_{\xi,v}(y) 
&=
\prod_{j=1}^{n}
\left(\enumber^{\al^2(y_j+v_j)^2}\,
\int_{\bR}
\enumber^{-2\pi\iu x_j\xi_j}\,
\enumber^{-\al^2 x_j^2-2\iu\al^2 x_j(y_j+v_j)}\, \dif x_j
\right)
\\
&= \frac{\pi^{n/2}}{\al^n}
\prod_{j=1}^{n} \enumber^{-\frac{\pi^2\xi_j^2}{\al^2}-2\pi(y_j+v_j)\xi_j}
= \frac{\pi^{n/2}}{\al^n}
\enumber^{-\frac{\pi^2\|\xi\|^2}{\al^2}
-2\pi\langle y+v,\xi\rangle}.
\end{align*}
In this example, $\Om=\bR^n$.
For $b$ in $L^\infty(\bR^n)$, we have
\[
\phi_b(x,y,v) 
=
\frac{\pi^{n/2}}{\al^n}
\int_{\bR^n} b(\xi)\,
\enumber^{2\pi\iu\,\langle x,\xi\rangle
-2\pi \langle y+v,\xi\rangle
-\frac{\pi^2 \|\xi\|^2}{\al^2}}\,\dif\nu(\xi)
\]
If $z=x+\iu y\in\bC^n$ and $w=u+\iu v\in\bC^n$,
then
\begin{align}
\notag
\phi_b(x-u,y,v) 
&=
\frac{\pi^{n/2}}{\al^n}
\int_{\bR^n} b(\xi)\,
\enumber^{2\pi\iu\,(\langle x-u,\xi\rangle
+ \iu\,\langle y+v,\xi\rangle)
-\frac{\pi^2 \|\xi\|^2}{\al^2}}\,\dif\nu(\xi)
\\
\label{eq:phi_Steinwart}
&=
\frac{\pi^{n/2}}{\al^n}
\int_{\bR^n} b(\xi)\,
\enumber^{2\pi\iu\,\langle z-\conj{w},\xi\rangle
-\frac{\pi^2 \|\xi\|^2}{\al^2}}\,\dif\nu(\xi).
\end{align}
In the last expression, we use the canonical inner product in $\bC^n$, linear with respect to the first argument.
Therefore,
for every $f$ in $H$ and every $z=x+\iu y$ in $\bC^n$,
\begin{equation}
\label{eq:S_phi_Steinwart}
(V_b f)(z)
=
\frac{2^n\al^n}{\pi^{n/2}}\,
\int_{\bC^n} f(w)\,
\Bigl(\,
\int_{\bR^n} b(\xi)\,
\enumber^{2\pi\iu\,\langle z-\conj{w},\xi\rangle
-\frac{\pi^2 \|\xi\|^2}{\al^2}}\,\dif\nu(\xi)
\Bigr)\,
\enumber^{-4\al^2 \|\Im(w)\|^2}\,\dif\mu_{\bC^n}(w).
\end{equation}
\end{example}

\smallskip

\begin{rem}[Applying a change of variables with weight]
\label{rem:change_of_variables}
In the rest of the examples in this section,
we use some weighted changes of variables
to transform the domain,
the Hilbert space,
and the operator algebra.
Namely, we suppose that $\cH_1$ is an RKHS of functions on a certain domain $\cD$,
with reproducing kernel $(K_z^{\cH_1})_{z\in\cD_1}$,
$\tau$ is a unitary representation of $G$ on $\cH_1$,
and our object of study is the von Neumann algebra $\cC(\tau)$.
We construct an auxiliary set $Y$,
a measure $\la$ on $Y$,
functions $\phi\colon G\times Y\to\cD_1$
and $p\colon G\times Y\to\bC$,
and a linear isometry
$\tU\colon\cH_1\to\cL^2(G\times Y)$,
acting by
\begin{equation}
\label{eq:weighted_change_of_variables}
(\tU f)(u,v) = p(u,v) f(\phi(u,v))\qquad(u\in G,\ v\in Y,\ f\in\cH_1).
\end{equation}
Then, it is easy to verify
(see~\cite[Proposition 9.6]{HerreraMaximenkoRamos2022}
or similar facts in
Paulsen and Raghupathi~\cite[Sections~5.6 and 5.7]{PaulsenRaghupathi2016}
or Saitoh and Sawano~\cite[Theorem~2.9 and Corollary~2.5]{SaitohSawano2016})
that $H\eqdef\tU(\cH_1)$ is an RKHS on $G\times Y$ with reproducing kernel
\[
K_{x,y}^H(u,v)
= \conj{p(x,y)} K_{\phi(x,y)}^{\cH_1}(\phi(u,v)) p(u,v).
\]
Let $U\colon\cH_1\to H$ be the compression of $\tU$.
Then, $U$ is an isometric isomorphism of Hilbert spaces.
Suppose that $G$, $Y$, $H$, $K$ satisfy the assumptions of this paper, and $U$ intertwines the unitary representations $\tau$ and $\rho$:
\begin{equation}
\label{eq:U_intertwines_the_unitary_representations}
U \tau(a) = \rho(a) U\qquad(a\in G).
\end{equation}
In this situation,
we apply the main results of this paper to $H$ and $\cC(\rho)$,
and then we pass from $\cC(\rho)$ to $\cC(\tau)$ by
$\cC(\tau)=U^\ast \cC(\rho) U$.
We know that each element of $\cC(\rho)$
is of the form $V_b=S_{\phi_b}$
for some $b$ in $L^\infty(\Om)$.
Therefore,
each element of $\cC(\tau)$
is of the form
$A_b \eqdef U^\ast V_b U$
for some $b$ in $L^\infty(\Om)$.
\end{rem}

\begin{example}[Radial operators in the Bergman space on the unit disk]
\label{Example_radial_Bergman_disk}
This example is based on~\cite[Example~9.7]{HerreraMaximenkoRamos2022}.
Let $\bD$ be the open unit disk in $\bC$,
with the restricted Lebesgue measure $\mu_\bD$;
notice that $\mu_\bD(\bD)=\pi$.
The Bergman space $\cH_1=L_{\mathrm{hol}}^2(\bD)$
consists of all square-integrable
holomorphic functions.
It is an RKHS with reproducing kernel given by
\[
K^{\cH_1}_w(z)
= K^{\cH_1}(z,w) 
= \frac{1}{\pi(1-z\overline{w})^2}.
\]
We denote by $G$ the group $\bR/(2\pi\bZ)$
which can be identified with the unit circle $\bT$.
In this paper,
we prefer to identify $G$ with $[0,2\pi)$,
where the group operation is addition modulo $2\pi$
and the topology is taken from $\bT$.
Let $\nu$ be the normalized Haar measure on $G$,
i.e., the restricted Lebesgue measure on $[0,2\pi)$
divided by $2\pi$.
For every $\theta$ in $[0,2\pi)$,
we denote by $\tau(\theta)$ the corresponding rotation operator:
\[
(\tau(\theta)f)(z) \eqdef f(\enumber^{-\iu\theta}z)
\qquad(f\in L_{\mathrm{hol}}^2(\bD)).
\]
Then, $\tau$ is a unitary representation of $G$ in $L_{\mathrm{hol}}^2(\bD)$.
The elements of its centralizer $\cC(\tau)$ are known as \emph{radial operators}.

Let $\hG=\bZ$ with counting measure $\hnu$.
The duality pairing between $G$ and $\bZ$ is given by
$E(\theta+2\pi\bZ,k) = \enumber^{\iu k\theta}$.
Let $Y = [0,1)$ with measure $\dif\la(r) = r\,\dif{}r$.
We define $\varphi\colon G\times Y\to\bD$ by
$\varphi(\theta,r)\eqdef r\,\enumber^{\iu\theta}$
and $p\colon G\times Y \to \bC$ by $p(\theta,r) \eqdef \sqrt{2\pi}$.
Consider $\cH_2=\cL^2([0,2\pi)\times[0,1), \nu\times\la)$.
The map $\tU$ defined by~\eqref{eq:weighted_change_of_variables} is a linear isometry, and the space $H\eqdef\tU(\cH_1)$ is an RKHS with reproducing kernel
\[
K_{\eta,s}(\theta,r)
= \frac{2}{\bigl(1-r s\enumber^{\iu\,(\theta-\eta)}\bigr)^2}.
\]
It is easy to verify that~\eqref{eq:U_intertwines_the_unitary_representations} holds; that is, $U$ intertwines rotations with horizontal translations.
Since $K_{0,s}(\cdot,r)$ expands into the Fourier series
\[
K_{0,s}(\theta,r) = 2\sum_{k=0}^{\infty} (k+1) (rs)^k \enumber^{\iu k\theta},
\]
its $k$th Fourier coefficient is
\[
L_{k,s}(r) = 2(k+1)(rs)^k \charfun_{\bNz}(k).
\]
In this example, $\Om = \bNz$.
For every $b$ in $l^\infty(\bNz)$,
\begin{equation}
\label{eq:phi_radial_Bergman_disk}
\phi_b(\theta,r,s)
= \sum_{k=0}^{\infty} 2b_k (k+1) r^k s^k \enumber^{\iu k\theta}.
\end{equation}
Therefore,
$V_b=S_{\phi_b}$
takes the following form:
\begin{align*}
(V_b f)(\theta,r)
&=
\frac{1}{2\pi}
\int_{\left[0,2\pi\right)\times \left[0,1\right)}
f(\eta,s) \psi(\theta-\eta,r,s) s\,\dif s\,\dif\eta
\\
&=
\frac{1}{\pi}
\int_{\left[0,2\pi\right)\times \left[0,1\right)} f(\eta,s)
\Bigl(\;
\sum_{k=0}^{\infty}b_k (k+1) r^k s^k\enumber^{\iu k(\theta-\eta)}
\Bigr)\,
s\,\dif s\,\dif\eta
\\
&=
\frac{1}{\pi}
\int_{\left[0,2\pi\right)\times \left[0,1\right)} f(\eta,s)
\Bigl(\;
\sum_{k=0}^{\infty}b_k (k+1) (r\enumber^{\iu\theta})^k (s\enumber^{-\iu\eta})^k
\Bigr)\,
s\,\dif s\,\dif\eta.
\end{align*} 
The corresponding radial operator
$A_b\eqdef U^\ast V_b U$,
having the same eigenvalues' sequence $b$,
acts in $L_{\mathrm{hol}}^2(\bD)$ by
\begin{equation}
\label{eq:radial_operators_integral_form}
(A_b f)(z)
=
(U^\ast V_b U f)(z)
=
\frac{1}{\pi}
\int_{\mathbb{D}} f(w)
\Bigl(\;\sum_{k=0}^{\infty}b_k (k+1) (z\overline{w})^k \Bigr)\,
\dif\mu_\bD(w).
\end{equation}
This representation was obtained in~\cite{Mohan-Venku_2024}.
\end{example}

Similarly to Example~\ref{Example_radial_Bergman_disk},
one can also obtain an integral representation for the radial operators on the Fock space $F^2(\bC)$ and for the separately radial operators on the Fock space $F^2(\bC^n)$,
repeating the results described in~\cite{BaisNaidu2023Fock}.

\begin{example}[Radial operators in the harmonic Bergman space over the unit disk]
\label{Example_radial_harmonic_Bergman_disk}
This example is based on
\cite[Example~9.8]{HerreraMaximenkoRamos2022};
see also~\cite{LoaizaLozano2013}.
The harmonic Bergman space $\cH_1=L_{\mathrm{harm}}^2(\bD)$
consists of all harmonic functions belonging to $\cL^2(\bD)$.
It is an RKHS with reproducing kernel given by
\[
K^{\cH_1}_w(z) = K^{\cH_1}(z,w)
= \frac{1}{\pi(1-z\conj{w})^2} + \frac{1}{\pi(1-w\conj{z})^2} - 1.
\]
Analogously to Example~\ref{Example_radial_Bergman_disk},
we can study radial operators acting on
$L_{\mathrm{harm}}^2(\mathbb{D})$.
Proceeding in the same manner, we get
\[
L_{k,s}(r) = 2(\vert k\vert+1)(rs)^{\vert k\vert}
\qquad(k\in\bZ,\ 0\le r,s<1).
\]
Thus, here we have $\Om = \hG = \bZ$.
For each $b$ in $l^2(\bZ)$,
\begin{equation}
\label{Eqn_phi_radial_harmonic_Bergman_disk}
\phi_b(\theta,r,s)
= \sum_{k=-\infty}^{\infty} 2b_k(\vert k\vert+1)r^{\vert k\vert}s^{\vert k\vert} \enumber^{\iu k\theta}.
\end{equation}
The corresponding radial operator
$A_b \eqdef U^\ast V_b U$
acts in $L_{\mathrm{harm}}^2(\bD)$ by
\begin{equation}
\label{eq:Sphi_radial_harmonic_Bergman_disk}
(A_b f)(z)
= \frac{1}{\pi}
\int_{\bD} f(w)\, \Bigl(
b_0
+\sum_{k=1}^{\infty}2(k+1) \bigl(b_k (z\conj{w})^{k}
+b_{-k}(\conj{z}w)^{k}\bigr)
\Bigr)
\dif\mu_{\bD}(w). 
\end{equation}
\end{example}

\begin{example}[Angular operators in the Bergman space over the upper half-plane]
\label{Example_angular_Bergman_UHP}
This example is based on
\cite[Example~9.10]{HerreraMaximenkoRamos2022},
\cite[Sections~7.1 and 7.2]{Vasilevski2008book},
and~\cite{BaisNaidu2024Bergman}.
For every $a$ in $\bR$,
we define the dilation operator $\tau(a)$ on
$L_{\mathrm{hol}}^2(\Pi)$ by
\[
(\tau(a)f)(z) \eqdef \enumber^{-a} f(\enumber^{-a}z).
\]
$\tau$ is a unitary representation of $\bR$
on $L_{\mathrm{hol}}^2(\Pi)$.
The elements of $\cC(\tau)$ are called
\emph{angular operators} on $L_{\mathrm{hol}}^2(\Pi)$. 

Take $G=\bR$, $Y=(0,\pi)$,
$\nu=\hnu$ be the Lebesgue measure on $\bR$,
and $\la$ be the Lebesgue measure on $(0,\pi)$.
We apply the scheme from Remark~\ref{rem:change_of_variables},
with $\phi\colon\bR\times (0,\pi) \to \Pi$ defined by
$\phi(r,\theta) \eqdef \enumber^{r+\iu\theta}$
and $p\colon \bR\times (0,\pi) \to \bC$ defined by
$p(r,\theta) \eqdef \enumber^{r+\iu\theta}$.
The map
$\tU\colon L_{\mathrm{hol}}^2(\Pi) \to L^2(\bR\times (0,\pi))$
is a linear isometry.
Take $H=\tU(L_{\mathrm{hol}}^2(\Pi))$
and define $U\colon L_{\mathrm{hol}}^2(\Pi)\to H$
to be $\tU$ with restricted codomain.
Recall that $\rho(a)$ acts in $H$
by
$(\rho(a)h)(r,\theta)=h(r-a,\theta)$.
Then, $U$ intertwines $\tau$ with $\rho$.
Furthermore,
$H$ is an RKHS with reproducing kernel
\begin{equation*}
K_{s,\eta}(r,\theta) 
=  - \frac{\enumber^{r+s} \enumber^{\iu\,(\theta-\eta)}}%
{\pi(\enumber^{r+ \iu\theta}-\enumber^{s-\iu\eta})^2} 
= -\frac{1}{4\pi\sinh\bigl(\frac{r-s+\iu\,(\theta+\eta)}{2}\bigr)^2}.
\end{equation*}
Similarly to~\cite[Example~9.10]{HerreraMaximenkoRamos2022}, we get
$\Om=\bR$ and
\begin{equation}
\label{eq:L_Bergman_angular}
L_{\xi,\eta}(\theta)
= 
-\frac{1}{4\pi}
\int_{\bR}
\frac{\enumber^{-2\pi\iu r\xi}\,\dif{}r}{\sinh\bigl(\frac{r+\iu\,(\theta+\eta)}{2}\bigr)^2}
=
\frac{4\pi\xi \enumber^{-2\pi\xi(\theta+\eta)}}{1-\enumber^{-4\pi^2\xi}}.
\end{equation}
Given $b$ in $L^\infty(\bR)$,
\begin{equation}
\label{eq:phi_angular_Bergman}
\phi_b(r,\theta,\eta)
=
\int_{\bR}
b(\xi)
\frac{4\pi\xi 
\enumber^{2\pi\iu\xi r}
\enumber^{-2\pi\xi(\theta+\eta)}}{1-\enumber^{-4\pi^2\xi}}\,\dif{}\xi
=
\int_{\bR}
b(\xi)
\frac{4\pi\xi 
\enumber^{2\pi\iu\xi(r+\iu\theta+\iu\eta)}}{1-\enumber^{-4\pi^2\xi}}\,\dif{}\xi.
\end{equation}
Thereby, we get an explicit integral form of the operator $V_b=S_{\phi_b}$ acting in $H$:
\begin{equation}
\label{eq:S_phi_angular_Bergman}
(V_b f)(r,\theta)
= 
\int_{\bR\times(0,\pi)}
f(s,\eta)
\biggl(\int_{\bR} b(\xi) \frac{4\pi\xi\,\enumber^{2\pi\iu\xi(r+\iu\theta-s+\iu\eta)}}{1-\enumber^{-4\pi^2\xi}}\,\dif\xi \biggr)\,\dif{}s \dif\eta.
\end{equation}
Every element of $\cC(\tau)$
has form
$A_b \eqdef U^\ast V_b U$
for some $b$ in $L^\infty(\bR)$.
We are going to compute $A_b$ explicitly.
For $f$ in $L_{\mathrm{hol}}^2(\Pi)$
and $z=\enumber^{r+\iu\theta}$ in $\Pi$,
\begin{align*}
(A_b f)(z)
&=
\enumber^{-r-\iu\theta}
\int_{\bR\times(0,\pi)}
\enumber^{s+\iu\eta}
f(\enumber^{s+\iu\eta})
\biggl(\int_{\bR} b(\xi)
\frac{4\pi \xi\,\enumber^{2\pi\iu\xi(r+\iu\theta-s+\iu\eta)}}{1-\enumber^{-4\pi^2\xi}}\,\dif\xi \biggr)\,\dif{}s \dif\eta
\\[1ex]
&=
\enumber^{-2(r+\iu\theta)}
\int_{\bR\times(0,\pi)}
f(\enumber^{s+\iu\eta})
\biggl(\int_{\bR} b(\xi) 
\Bigl(\frac{\enumber^{r+\iu\theta}}{\enumber^{s-\iu\eta}}\Bigr)^{2\pi\iu\xi+1}
\frac{4\pi\xi}{1-\enumber^{-4\pi^2\xi}}\,\dif\xi \biggr)\,
\enumber^{2s}\,\dif{}s \dif\eta.
\end{align*}
On substituting $e^s = \alpha$, we get  
\begin{align*}
(A_b f)(z)
&=
\enumber^{-2(r+\iu\theta)}
\int_{\bR_+\times(0,\pi)}
f(\alpha\enumber^{\iu\eta})
\biggl(\int_{\bR} b(\xi) 
\Bigl(\frac{\enumber^{r+\iu\theta}}{\alpha\enumber^{-\iu\eta}}\Bigr)^{2\pi\iu\xi+1}
\frac{4\pi\xi}{1-\enumber^{-4\pi^2\xi}}\,\dif\xi \biggr)\,
\alpha\,\dif\alpha \dif\eta.
\end{align*}
After the change of variables
$w=\alpha\enumber^{\iu\eta}$
in the outer integral,
\begin{equation}
\label{eq:operator_angular_Bergman_complex_form}
(A_b f)(z)
=
\frac{1}{z^2}
\int_{\Pi}
f(w)
\biggl(\int_{\bR} b(\xi) 
\left(\frac{z}{\conj{w}}\right)^{2\pi\iu\xi+1}
\frac{4\pi\xi}{1-\enumber^{-4\pi^2\xi}}\,\dif\xi \biggr)\,\dif\mu_\bC(w).
\end{equation}
Finally, setting $\be=2\pi\xi$ in the inner integral, we obtain
\begin{equation}
\label{eq:operator_angular_Bergman_complex_form2}
(A_b f)(z)
=
\frac{1}{\pi z^2}
\int_{\Pi}
f(w)
\biggl(\int_{\bR}
b\left(\frac{\be}{2\pi}\right) 
\left(\frac{z}{\conj{w}}\right)^{\iu\be+1}
\frac{\be}{1-\enumber^{-2\pi\be}}\,\dif\be \biggr)\,\dif\mu_\bC(w).
\end{equation}
This is similar to the representation of angular operators given in \cite{BaisNaidu2024Bergman};
the unique difference is the dilation coefficient in $b$
which does not change the spectral properties of the operator.
\end{example}

\begin{example}[Vertical operators on the Fock space]
Let $\ga>0$.
The Bargmann--Segal--Fock space $F=F_\ga^2(\bC^n)$
consists of all holomorphic functions $f$ such that
\[
\Vert f\Vert_{F}
\eqdef
\frac{\ga^n}{\pi^n}
\int_{\bC^n}
|f(z)|^2\,\enumber^{-\ga\|z\|^2}\,\dif\mu_{\bC^n}(z)
< +\infty.
\]
It is an RKHS with reproducing kernel
$K^F(z,w) = K^F_w(z) = \enumber^{\ga \langle z,w\rangle}$.
For every $a$ in $\bR^n$,
we denote by $\tau(a)$ the ``horizontal''
Weyl translation operator, defined by
\[
(\tau(a) f)(z)
\eqdef
f(z-a) \enumber^{\ga \langle z,a\rangle - \frac{\ga \|a\|^2}{2}}.
\]
It is easy to see that $\tau$ is a unitary representation of $\bR^n$ in $F$.
For the elements of the centralizer $\cC(\tau)$,
we use the term ``vertical operators on $F$''.
Let $G$, $Y$, $\nu=\hnu$, $\la$, $\rho$, and $H$
be as in Example~\ref{Example_Vertical_RBFK},
with $\al=\sqrt{\ga/2}$.
Define $U\colon F \to H$ by
\[
(Uf)(z)\eqdef
\enumber^{-\al^2
\sum_{j=1}^n z_j^2}\,
f(z)
=
\enumber^{-\frac{\ga}{2}
\sum_{j=1}^n z_j^2}\,
f(z).
\]
It is easy to verify that $U\tau(a)=\rho(a)U$
for every $a$ in $\bR^n$.
In other words, $U$ intertwines horizontal Weyl translations
(acting on $F$)
with usual horizontal translations (acting on $H$).
Therefore, $\cC(\tau)=U^\ast \cC(\rho) U$.

Given $b$ in $L^\infty(\bR^n)$,
we denote by $V_b$ the operator~\eqref{eq:S_phi_Steinwart}
from Example~\ref{Example_Vertical_RBFK},
and set $A_b \eqdef U^\ast V_b U$.
Every element of $\cC(\tau)$ is of the form $A_b$
for some $b$ in $L^\infty(\bR^n)$.
Let us compute $A_b$ explicitly.
To simplify notation,
we assume $\ga=1$, $\al=\sqrt{1/2}$, and $n=1$.
Then,
\[
(V_b f)(z)
=
\frac{\sqrt{2}}{\sqrt{\pi}}\,
\int_{\bC} f(w)\,
\Bigl(\,
\int_{\bR} b(\xi)\,
\enumber^{2\pi\iu\,(z-\conj{w})\xi
-2\pi^2 \xi^2}\,\dif\nu(\xi)
\Bigr)\,
\enumber^{-2 \Im(w)^2}\,\dif\mu_{\bC}(w),
\]
and
\begin{align*}
(A_b f)(z)
&= (U^\ast V_b U f)(z)
\\
&=
\enumber^{\frac{z^2}{2}}
\frac{\sqrt{2}}{\sqrt{\pi}}\,
\int_{\bC} f(w)\,
\Bigl(\,
\int_{\bR} b(\xi)\,
\enumber^{2\pi\iu\,(z-\conj{w})\xi
-2\pi^2 \xi^2}\,\dif\nu(\xi)
\Bigr)\,
\enumber^{-2 \Im(w)^2}\,
\enumber^{-\frac{w^2}{2}}\,
\dif\mu_{\bC}(w)
\\[0.5ex]
&=
\frac{\sqrt{2}}{\sqrt{\pi}}\,
\int_{\bC} f(w)\,
\Bigl(\,
\int_{\bR} b(\xi)\,
\enumber^{-2
\left(\pi\xi - \frac{\iu\,(z-\conj{w})}{2}\right)^2}
\,\dif\nu(\xi)
\Bigr)\,
\enumber^{z\conj{w}}\,
\enumber^{-2 \Im(w)^2}\,
\enumber^{-\frac{w^2}{2} - \frac{\conj{w}^2}{2}}\,
\dif\mu_{\bC}(w).
\end{align*}
After a simplification,
\begin{equation}
\label{eq:T_b_Fock}
(A_b f)(z)
=
\frac{\sqrt{2}}{\pi\sqrt{\pi}}\,
\int_{\bC} f(w)\,
\Bigl(\,
\int_{\bR}
b\left(\frac{\be}{\pi}\right)\,
\enumber^{-2
\left(\be - \frac{\iu\,(z-\conj{w})}{2}\right)^2}
\,\dif\nu(\be)
\Bigr)\,
\enumber^{z\conj{w}}\,
\enumber^{-|w|^2}\,
\dif\mu_{\bC}(w).
\end{equation}
This is similar to the class of integral operators studied in~\cite{BaisNaidu2023Fock},
with the following small differences:
1) here we consider ``vertical'' operators instead of ``horizontal'' ones;
2) here we have $b(\beta/\pi)$ instead of $b$.
Using this representation, one can also obtain boundedness results studied in~\cite{Cao_SingInt_AdvMath_2020} for a class of singular integral operators on the Fock space $F^2(\bC^n)$
that settles the open question posed by Zhu in~\cite{Zhu_SingInt_IEOT_2015}.
\end{example}

\section*{Funding}

The first author thanks The Institute of Mathematical Sciences (IMSc), India for providing research facilities and financial support.
The second author has been supported by
SECIHTI (Mexico) project ``Ciencia de Frontera''
FORDECYT-PRONACES/61517/2020
and by IPN-SIP projects
(Instituto Polit\'{e}cnico Nacional, Mexico). The third author was supported by SERB-MATRICS, Government of India, under the grant MTR/2023/000750.

\bigskip\noindent
Shubham R. Bais\newline
The Institute of Mathematical Sciences \newline
A CI of Homi Bhabha National Institute \newline
CIT Campus, Taramani, Chennai \newline
India 600 113. \newline
email: shubhambais007@gmail.com, shubhambais@imsc.res.in \newline
https://orcid.org/0009-0000-1760-1564

\bigskip\noindent
Egor A. Maximenko\newline
Instituto Polit\'{e}cnico Nacional\newline
Escuela Superior de F\'{i}sica y Matem\'{a}ticas\newline
Ciudad de M\'{e}xico, 07738, Mexico\newline
email: egormaximenko@gmail.com, emaximenko@ipn.mx\newline
https://orcid.org/0000-0002-1497-4338

\bigskip\noindent
D. Venku Naidu\newline
Indian Institute of Technology Hyderabad\newline
Department of Mathematics,\newline
Kandi, Sangareddy,
Telangana \newline
India 502285\newline
email:
venku@math.iith.ac.in\newline
https://orcid.org/0000-0002-1279-0105

\end{document}